\documentclass[12pt,reqno,twoside]{article}
\usepackage{amssymb,amsfonts,amsthm,amsmath}
\usepackage{cite}
\usepackage[left=.9 in,top=.9 in,right=.9 in,bottom=.9 in]{geometry}
\usepackage{enumitem}

\usepackage{hyperref}

\newcommand{\beq}{\begin{equation}}
\newcommand{\eeq}{\end{equation}}
\newcommand{\beqs}{\begin{equation*}}
\newcommand{\eeqs}{\end{equation*}}
\newcommand{\ba}{\begin{array}}
\newcommand{\ea}{\end{array}}
\newcommand{\beas}{\begin{eqnarray*}}
\newcommand{\eeas}{\end{eqnarray*}}
\newcommand{\bea}{\begin{eqnarray}}
\newcommand{\eea}{\end{eqnarray}}
\newcommand{\bal}{\begin{align}}
\newcommand{\eal}{\end{align}}

\newcommand{\bals}{\begin{align*}}
\newcommand{\eals}{\end{align*}}

\newcommand{\al}{\alpha}

\newcommand{\R}{\ensuremath{\mathbb R}}
\newcommand{\C}{\ensuremath{\mathbb C}}
\newcommand{\N}{\ensuremath{\mathbb N}}
\newcommand{\Z}{\ensuremath{\mathbb Z}}

\newcommand{\inprod}[1]{\langle{#1}\rangle}
\newcommand{\doubleinprod}[1]{\langle\!\langle{#1}\rangle\!\rangle}

\newcommand{\bds}{\begin{displaystyle}}
\newcommand{\eds}{\end{displaystyle}}

\def\eqdef{\stackrel{\rm def}{=}}
\def\expsim{\stackrel{\rm exp.}{\sim}}
\def\polsim{\sim}

\newcommand{\bvec}[1]{\mathbf{#1}}
\def\vecx{\bvec x}

\def\vece{\bvec e}
\def\vecu{\bvec u}
\def\vecv{\bvec v}
\def\vecx{\bvec x}
\def\vecw{\bvec w}

\def\veck{\bvec k}

\def\varep{\varepsilon}
\def\ddt{\frac{d}{dt}}

\newtheorem{theorem}{Theorem}[section]

\newtheorem{lemma}[theorem]{Lemma}
\newtheorem{corollary}[theorem]{Corollary}

\newtheorem{definition}[theorem]{Definition}

\theoremstyle{remark}
\newtheorem{remark}[theorem]{\bf{Remark}}
\newtheorem{example}[theorem]{\bf{Example}}

\def\mD{\mathcal D}

\numberwithin{equation}{section}

\title{\textbf{Long-time asymptotic expansions for Navier-Stokes equations with power-decaying forces}}
\author{Dat Cao and Luan Hoang}

\date{\today}
\begin{document}
\maketitle

\begin{center}
Department of Mathematics and Statistics, Texas Tech University\\
Box 41042, Lubbock, TX 79409-1042, U.S.A.\\
Email addresses: \texttt{dat.cao@ttu.edu, luan.hoang@ttu.edu}

\end{center}

\begin{abstract}
The Navier-Stokes equations for viscous, incompressible fluids are studied in the three-dimensional periodic domains,
with the body force having an asymptotic expansion, when time goes to infinity, in terms of power-decaying functions in a Sobolev-Gevrey space. Any Leray-Hopf weak solution is proved to have an asymptotic expansion of the same type in the same space, which is uniquely determined by the force, and independent of the individual solutions. In case the expansion is convergent, we show that the next asymptotic approximation for the solution must be an exponential decay. 
Furthermore, the convergence of the expansion and the range of its coefficients, as the force varies, are investigated.
\end{abstract}
\section{Introduction}\label{intro}

The Navier-Stokes equations (NSE) are non-linear partial differential equations that describe the dynamics of viscous, incompressible fluids. The mathematics of NSE has proven to be quite important, intriguing and challenging. 
In particular, understanding the long-term behaviors of the solutions of NSE would be insightful to many hydrodynamical phenomena.
Unfortunately, such a level of mathematical understanding is still not available in general.
However, under some circumstances, the mathematics is more accessible and much has been understood. One such case is when the body force is potential, which has many papers devoted to \cite{DE1968,FS83,FS84a,FS84b,FS87,FS91,FHOZ1,FHOZ2,FHN1,FHN2,FHS1}. 
(We caution that these works provide deep understanding of the solutions despite the fact that they go to zero as time becomes large. They are different from those studying the case of large forces, which are more oriented toward the theory of turbulence.) 
The case when the force is non-potential and decays exponentially in time  has only been studied recently in \cite{HM2}. 
The current paper follows this direction of research. We aim to understand the long-term behavior of the solutions in case the force is larger than those considered in \cite{HM2}. More importantly, we hope to find new phenomena due to the different structure of the force, and describe precisely how  the asymptotic properties of the force determine the asymptotic behavior of the solutions.

First, we recall the mathematical formulation of the NSE, and specify the context in which we study them.
Let $\vecu(\vecx,t)\in\R^3$ denote the  velocity vector field and  $p(\vecx,t)\in\R$ denote the pressure of a viscous, incompressible fluid, where  $\vecx\in \R^3$ is the vector of spatial variables, and $t\in\R$ is the time variable. 
The (kinematic) viscosity of the fluid is a constant $\nu>0$.
The body force acting on the fluid is $\mathbf f(\vecx,t)\in\R^3$. The NSE are the following system of partial differential equations
\begin{align}\label{nse}
\begin{split}
&\bds \frac{\partial \vecu}{\partial t}\eds  + (\vecu\cdot\nabla)\vecu -\nu\Delta \vecu = -\nabla p+\mathbf f \quad\text{on }\R^3\times(0,\infty),\\
&\textrm{div } \vecu = 0  \quad\text{on }\R^3\times(0,\infty).
\end{split}
\end{align}
Above, the first equation is the balance of momentum, while the second one is the incompressibility condition.

The initial condition specified for the velocity is 
\beq\label{ini}
\vecu(\vecx,0) = \vecu^0(\vecx),
\eeq 
where  $\vecu^0(\vecx)$ is a given divergence-free vector field. 
 
In our current study, the force $\mathbf f(\vecx,t)$ and solutions $(\vecu(\vecx, t),p(\vecx, t))$ are considered to belong to the class of $L$-periodic functions for some $L>0$, that is, the class of functions $g(\vecx)$ that satisfy
\beqs
g(\vecx+L\vece_j)=g(\vecx)\quad \textrm{for all}\quad \vecx\in \R^3,\ j=1,2,3,\eeqs
where  $\{\vece_1,\vece_2,\vece_3\}$ is the standard basis of $\R^3$.
Such a consideration will simplify our mathematical analysis since it avoids the case of unbounded domains, and the no-slip boundary condition usually imposed on bounded domains. 

By a Galilean transformation, see e.g. \cite{HM2}, we can assume that $\mathbf f(\vecx,t)$ and $\vecu(\vecx,t)$, for all $t\ge 0$, have zero averages over the domain $\Omega=(-L/2, L/2)^3$, that is, their spatial integrals over $\Omega$ are zero.

Thanks to the Leray-Helmholtz decomposition, and for the sake of convenience, we assume further that $\mathbf f(\vecx,t)$ is divergence-free for all $t\ge 0$. 

By rescaling the variables $\vecx$ and $t$, we assume throughout, without loss of generality, that  $L=2\pi$ and $\nu =1$.
With this assumption, the equations in \eqref{nse} are adimensional now.

\medskip
In studying the dynamics of NSE, the function $\vecu(\vecx,t)$ of several variables can be viewed as a function of $t$ valued in some functional space.
For time-dependent functions of such type,  their  asymptotic properties, as time goes to infinity, can be understood most precisely if some form of asymptotic expansions is established. We discuss, in this paper, the following two types of expansions. Briefly speaking, one expansion is in terms of exponential decaying functions with polynomial coefficients, and the other of power-decaying ones.

\begin{definition}\label{expanddef}
Let $(X, \|\cdot\|)$ be a real normed space, $g$ be a function from $(0,\infty)$ to $X$, and $(\gamma_n)_{n=1}^\infty$ be a strictly increasing, divergent sequence of positive numbers.

{\rm (a)} The function $g$ is said to have the asymptotic expansion
\beq \label{gexpand}
g(t) \expsim \sum_{n=1}^\infty g_n(t)e^{-\gamma_nt} \text{ in } X,
\eeq
where $g_n(t)$'s are $X$-valued polynomials in $t$, if for any $N\geq1$, there exists $\beta_N>\gamma_N$ such that
\beqs
\Big\|g(t)- \sum_{n=1}^N g_n(t)e^{-\gamma_n t}\Big\|=\mathcal O(e^{-\beta_N t})\ \text{as }t\to\infty.
\eeqs

{\rm (b)} The function $g$ is said to have the asymptotic expansion
	\beq \label{polexpand}
g(t) \polsim \sum_{n=1}^\infty \xi_n t^{-\gamma_n} \text{ in } X,
\eeq
where $\xi_n$'s are elements in $X$, if for any $N\geq1$, there exists $\beta_N>\gamma_N$ such that
\beqs 
\Big\|g(t)- \sum_{n=1}^N \xi_n t^{-\gamma_n}\Big\|=\mathcal O(t^{-\beta_N})\ \text{as }t\to\infty.
\eeqs

\end{definition}

Throughout the paper, we will make use of the following notation  
$$u(t)=\vecu(\cdot,t),\  f(t)=\mathbf f(\cdot,t),\  u^0=\vecu^0(\cdot).$$ 
Note that $u^0$, and each value of $u(t)$, $f(t)$ belong to some functional spaces.


\medskip
In case the force  $\mathbf f$ in NSE is a potential function, i.e., $\mathbf f(\vecx,t)=-\nabla \phi(\vecx,t)$ for some scalar function $\phi$, it is well-known that any Leray-Hopf weak solution becomes regular eventually and decays in $H^1(\Omega)$-norm exponentially.
The first precise asymptotic behavior is proved by Foias and Saut \cite{FS84a}. Namely, for any  non-trivial, regular solution
$u(t)$ in bounded or periodic domains, there exist an eigenvalue $\lambda$  of the Stokes operator and a corresponding eigenfunction $\xi$ such that
\beqs
\lim_{t\to\infty} e^{\lambda t}u(t)=\xi,   \text{ where the limit holds in all Sobolev norms.}
\eeqs

Moreover,  they showed in  \cite{FS87}  that any such solution admits an  asymptotic expansion 
\beq \label{expand}
u(t) \expsim \sum_{n=1}^\infty q_n(t)e^{-\mu_nt}
\eeq
in Sobolev spaces $H^m(\Omega)^3$ for all $m\ge 0$. Here, $\{\mu_n:n\in\N\}$ is the additive semigroup generated by the spectrum of the Stokes operator. It was then improved in \cite{HM1}, for the case of periodic domains, that the expansion holds in any Gevrey spaces $G_{\alpha,\sigma}$, see section \ref{Prelim} for details.

Studying the asymptotic expansion \eqref{expand} leads to theories of associated normalization map and invariant nonlinear manifolds \cite{FS83,FS84a,FS84b,FS87,FS91}, Poincar\'e-Dulac normal form for NSE  \cite{FHOZ1,FHOZ2,FHS1}; they were applied to analysis of helicity, statistical solutions of the NSE, and decaying turbulence \cite{FHN1,FHN2}. 
It  provides  fine details for the long-time behavior of the solutions, and sheds some insights into the nonlinear structure of NSE.
See  also \cite{KukaDecay2011} for a result in $\mathbb R^3$, \cite{Shi2000} for expansions for dissipative wave equations, and the survey paper \cite{FHS2} for more information on the subject.

\medskip
Regarding the problem of establishing the expansion \eqref{expand}, the simplified approach in \cite{HM1}, for NSE in the periodic domains, turns out to be easily adapted to the case of non-potential forces \cite{HM2}.
We recall here a result in this direction -- Theorem 2.2 of \cite{HM2}. 

\textit{Assume  there exists $\sigma\ge 0$, such that  
\beq\label{fexpand}
f(t)\expsim \sum_{n=1}^\infty f_n(t)e^{-nt} \text{ in $G_{\alpha,\sigma}$ for all $\alpha\ge 0$.} 
\eeq
Then any Leray-Hopf weak solution $u(t)$ of \eqref{nse} and \eqref{ini} admits an asymptotic expansion
\beq\label{uperi}
u(t)\expsim \sum_{n=1}^\infty q_n(t)e^{-nt} \text{ in $G_{\alpha,\sigma}$ for all $\alpha\ge 0$.} 
\eeq
}

\medskip
We note that the expansion \eqref{fexpand} of $f$ is of type \eqref{gexpand}, and so are the expansions \eqref{expand} and \eqref{uperi}. A natural question arising is whether one can establish the same results for other types of decaying forces.
This paper studies a particular case when $f$ has an asymptotic expansion of type \eqref{polexpand} instead. 
More specifically, assume  there exist $\alpha\ge 1/2$ and $\sigma\ge 0$ such that 
\beq\label{fpol}
f(t)\polsim \sum_{n=1}^\infty \psi_n t^{-\gamma_n}\quad\text{in } G_{\alpha,\sigma}.
\eeq

We will derive a corresponding expansion for solutions of NSE. 
First, rewrite \eqref{fpol} as
\beqs 
f(t)\polsim \sum_{n=1}^\infty \phi_n t^{-\mu_n }\quad\text{in } G_{\alpha,\sigma},
\eeqs
where $(\mu_n)_{n=1}^\infty$ is an appropriate sequence of powers generated by $\gamma_n$'s.

We prove that there exist $\xi_n\in G_{\alpha+1,\sigma}$, for all $n\in\N$, which are explicitly determined by $\phi_n$'s, such that any Leray-Hopf weak solution $u(t)$ will admit the following expansion
\beq\label{upol}
u(t)\polsim \sum_{n=1}^\infty \xi_n t^{-\mu_n} \quad\text{in } G_{\alpha+1-\rho,\sigma},\text{ for all }\rho\in (0,1).
\eeq

The expansion \eqref{upol} has the following new features.
\begin{enumerate}[label=\rm (\alph*)]
 \item All Leray-Hopf weak solutions have the \textit{same expansion}, depending only on the force, regardless even their uniqueness and global regularity.

 \item The expansion of solution $u$ is established with 
the force $f$ belonging to a Sobolev-Gevrey space $G_{\alpha,\sigma}$ for \textit{fixed} $\alpha$ and $\sigma$. This contrasts with the requirement in \eqref{fexpand} that $f\in G_{\alpha,\sigma}$ for \textit{all} $\alpha\ge 0$. 
\end{enumerate}

We also note that the force in \eqref{fpol}, though decays to zero, is much larger than the one in \eqref{fexpand}, as $t\to\infty$.

Although our proof follows the scheme installed in \cite{FS87} and \cite{HM1,HM2}, we take advantage of the new structure of the force $f$ to make significant improvements in estimates, and succeed in quantifying the effects of such structure on that of the solution $u$.

\medskip
Since the expansion \eqref{upol} is convergent in many cases, we investigate what may be the next approximation of the solution after this expansion.
Specifically, if $f(t)=\sum_{n=1}^\infty \phi_n t^{-n}$ and  $\bar u(t)\eqdef\sum_{n=1}^\infty \xi_n t^{-n}$ are  uniformly convergent in appropriate spaces for large $t$, then
$u(t)-\bar u(t)$ is proved to decay at least at the rate $t^{\beta}e^{-t}$, as $t\to\infty$, for some number $\beta\ge 0$.
This result rules out any intermediate approximation of $u$ after $\bar u$ that is between the power and exponential decays.

\medskip
The paper is organized as follows. 
Section \ref{Prelim} reviews the functional setting for NSE, some basic inequalities for Sobolev and Gevrey norms, and estimates for the bi-linear form $B(u,v)$ in NSE.
Lemma \ref{ineqlem} describes the asymptotic behavior of certain integrals which will be utilized repeatedly in asymptotic estimates for large time.
Particularly, it is used in Lemma \ref{odelem} to establish the limit, as $t\to\infty$,  and the remainder estimates for solutions of certain linearized NSE.
This will be a building block of proving the asymptotic expansion \eqref{upol}.
In section \ref{Gdecay}, we establish the power-decay for any Leray-Hopf weak solutions, cf. Theorem \ref{theo23}. It combines standard energy estimates, when time is large, with Theorem \ref{theo22}, which proves strong  asymptotic bounds for solutions in Gevrey spaces when the initial data and the force are small. 
In section \ref{expsec}, the asymptotic expansion \eqref{upol} is obtained, either as a finite sum in Theorem \ref{finitetheo}, or an infinite sum in Theorem \ref{mainthm}.
As mentioned in remarks (a) and (b) above, the same expansion holds for \textit{all} Leray-Hopf weak solutions, and only requires the force $f$ to belong to a \textit{fixed} Sobolev-Gevrey space, namely, $G_{\alpha,\sigma}$. Moreover, the expansion of the solution $u$ holds in even more regular space, $G_{\alpha+1-\rho,\sigma}$, than that of $f$.
This feature is possible because of the higher regularity for the elements $\xi_n$'s in Lemma \ref{xireg}, and the remainder estimate in Lemma \ref{odelem}. 
It is also worth mentioning that the $\xi_n$'s are explicitly determined by the recursive formulas \eqref{xi1} and \eqref{xin1} without solving any ordinary differential equations (in functional spaces) which was the case for the expansions \eqref{expand} and \eqref{uperi}.
Section \ref{furthersec} deals with the convergence of the expansions, and the range of $\xi_n$'s as the force varies.
In case $\gamma_n=\mu_n=n$ for all $n$, it turns out that the expansion \eqref{upol} can be any finite sum, or an infinite sum with the norms $\|\xi_n\|_{G_{\alpha+1,\sigma}}$ decaying in a certain, but still very general, way, see Theorem \ref{conv-series}, Example \ref{cn} and Corollary \ref{finite}.
Since the sequence $(\xi_n)_{n=1}^\infty$ completely determines the asymptotic expansion \eqref{upol}, it plays a similar role to the normalization map $W$ in \cite{FS87,FS91}. Therefore, a few comparisons between them are made in Remark \ref{normalcompare}.
Another topic in this section is to find out what will be the next approximation of the solution $u$ after the expansion \eqref{upol}. It is proved in Theorem \ref{nextrate} that, in case the expansion converges, say, to $\bar u$, the remainder $u-\bar u$ must decay exponentially. 

\section{Preliminaries}\label{Prelim}

\subsection{Background for NSE}
Let $L^2(\Omega)=H^0(\Omega)$ and $H^m(\Omega)=W^{m,2}(\Omega)$, for integers $m\ge 0$, denote the standard Lebesgue and Sobolev spaces on $\Omega$.
The standard inner product and norm in $L^2(\Omega)^3$ are denoted by $\inprod{\cdot,\cdot}$ and $|\cdot|$, respectively.
(We warn that this  notation  $|\cdot|$ also denotes the Euclidean norm in $\R^n$ and $\C^n$, for any $n\in\N$, but its meaning will be clear based on the context.)

Let $\mathcal{V}$ be the set of all $2\pi$-periodic trigonometric polynomial vector fields which are divergence-free and  have zero average over $\Omega$.  
Define
$$H, \text{ resp. } V\ =\text{ closure of }\mathcal{V} \text{ in }
L^2(\Omega)^3, \text{ resp. } H^1(\Omega)^3.$$
Notice that each element of $H$ is divergence-free and has zero average over $\Omega$,
 and each element of $V$ is $2\pi$-periodic.

We use the following embeddings and identification
$$V\subset H=H'\subset V',$$ 
where each space is dense in the next one, and the embeddings are compact.


Let $\mathcal{P}$ denote the orthogonal (Leray) projection in $L^2(\Omega)^3$ onto $H$.

The Stokes operator $A$ is a bounded linear mapping from $V$ to its dual space $V'$ defined by  
\beqs
\inprod{A\vecu,\vecv}_{V',V}=
\doubleinprod{\vecu,\vecv}
\eqdef \sum_{j=1}^3 \inprod{ \frac{\partial \vecu}{\partial x_j} , \frac{\partial \vecv}{\partial x_j} }\quad \text{for all } \vecu,\vecv\in V.
\eeqs

As an unbounded operator on $H$, the operator $A$ has the domain $\mD(A)=V\cap H^2(\Omega)^3$, and, under  the current consideration of periodicity conditions, 
\beqs A\vecu = - \mathcal{P}\Delta \vecu=-\Delta \vecu\in H \quad \textrm{for all}\quad \vecu\in\mD(A).
\eeqs

The spectrum  of $A$ is known to be 
$$\sigma(A)=\{|\veck|^2: \ \veck\in\Z^3, \veck\ne \mathbf 0\},$$
and each $\lambda\in \sigma(A)$ is an eigenvalue.
Note that $\sigma(A)\subset \N$ and $1\in\sigma(A)$, hence, the additive semigroup generated by $\sigma(A)$ is $\N$.

For $n\in\sigma (A)$, we denote by  $R_n$ the orthogonal projection in $H$ on the eigenspace of $A$ corresponding to $n$,
and set $$P_n=\sum_{j\in \sigma(A),j\le n}R_j.$$ 
Note that each vector space $P_nH$ is finite dimensional.

\medskip 
For $\alpha,\sigma \in \R$ and  $\vecu=\sum_{\veck\ne \mathbf 0} 
\widehat \vecu(\veck)e^{i\veck\cdot \vecx}$, define
$$A^\alpha \vecu=\sum_{\veck\ne \mathbf 0} |\veck|^{2\alpha} \widehat \vecu(\veck)e^{i\veck\cdot 
\vecx},\quad e^{\sigma A^{1/2}} \vecu=\sum_{\veck\ne \mathbf 0} e^{\sigma 
|\veck|} \widehat \vecu(\veck)e^{i\veck\cdot 
\vecx},$$
and, hence,
$$A^\alpha e^{\sigma A^{1/2}} \vecu=e^{\sigma A^{1/2}}A^\alpha  \vecu=\sum_{\veck\ne \mathbf 0} |\veck|^{2\alpha}e^{\sigma 
|\veck|} \widehat \vecu(\veck)e^{i\veck\cdot 
\vecx}.$$

The  Gevrey spaces are defined by
\beqs
G_{\alpha,\sigma}=\mD(A^\alpha e^{\sigma A^{1/2}} )\eqdef \{ \vecu\in H: |\vecu|_{\alpha,\sigma}\eqdef |A^\alpha 
e^{\sigma A^{1/2}}\vecu|<\infty\},
\eeqs
and, in particular, when $\sigma=0$,  the domain of the fractional operator $A^\al$ is 
\beqs 
\mD(A^\alpha)=G_{\alpha,0}=\{ \vecu\in H: |A^\alpha \vecu|=|\vecu|_{\alpha,0}<\infty\}.
\eeqs

Clearly, each space $G_{\alpha,\sigma}$ with the norm $|\cdot|_{\alpha,\sigma}$ is a Banach space.

Thanks to the zero-average condition, the norm $|A^{m/2}\vecu|$ is equivalent to $\|\vecu\|_{H^m(\Omega)^3}$ on the space $\mathcal D(A^{m/2})$ for $m=0,1,2,\ldots$

Note that $\mD(A^0)=H$,  $\mD(A^{1/2})=V$, and  $\|\vecu\|\eqdef |\nabla \vecu|$ is equal to $|A^{1/2}\vecu|$ for $\vecu\in V$.
Also, the norms $|\cdot|_{\alpha,\sigma}$ are increasing in $\alpha$, $\sigma$, hence, the spaces $G_{\alpha,\sigma}$ are decreasing in $\alpha$, $\sigma$.

\medskip
Regarding the nonlinear term in the NSE, a bounded linear map $B:V\times V\to V'$ is defined by
\beqs
\inprod{B(\vecu,\vecv),\vecw}_{V',V}=b(\vecu,\vecv,\vecw)\eqdef \int_\Omega ((\vecu\cdot \nabla) \vecv)\cdot \vecw\, d\vecx, \quad \textrm{for all}\quad \vecu,\vecv,\vecw\in V.
\eeqs 
In particular,
\beqs
B(\vecu,\vecv)=\mathcal{P}((\vecu\cdot \nabla) \vecv), \quad \textrm{for all}\quad \vecu,\vecv\in\mD(A).
\eeqs

The problems  \eqref{nse} and \eqref{ini} can now be rewritten in the functional form as
\begin{align}\label{fctnse}
&\frac{du(t)}{dt} + Au(t) +B(u(t),u(t))=f(t) \quad \text{ in } V' \text{ on } (0,\infty),\\
\label{uzero} 
&u(0)=u^0\in H.
\end{align}
(We refer the reader to the books \cite{LadyFlowbook69,CFbook,TemamAMSbook,TemamSIAMbook} for more details.)


The next definition makes precise the meaning of weak solutions of \eqref{fctnse}.

\begin{definition}\label{lhdef}
Let $f\in L^2_{\rm loc}([0,\infty),H)$.
A \emph{Leray-Hopf weak solution} $u(t)$ of \eqref{fctnse} is a mapping from $[0,\infty)$ to $H$ such that 
\beqs
u\in C([0,\infty),H_{\rm w})\cap L^2_{\rm loc}([0,\infty),V),\quad u'\in L^{4/3}_{\rm loc}([0,\infty),V'),
\eeqs
and satisfies 
\beq\label{varform}
\ddt \inprod{u(t),v}+\doubleinprod{u(t),v}+b(u(t),u(t),v)=\inprod{f(t),v}
\eeq
in the distribution sense in $(0,\infty)$, for all $v\in V$, and the energy inequality
\beq\label{Lenergy}
\frac12|u(t)|^2+\int_{t_0}^t \|u(\tau)\|^2d\tau\le \frac12|u(t_0)|^2+\int_{t_0}^t \langle f(\tau),u(\tau)\rangle d\tau
\eeq
holds for $t_0=0$ and almost all $t_0\in(0,\infty)$, and all $t\ge t_0$.  
Here, $H_{\rm w}$ denotes the topological vector space $H$ with the weak topology.
 
If a Leray-Hopf weak solution belongs to $C([0,\infty),V)$, it is called a regular solution.
\end{definition}

\medskip
We assume throughout the paper that 

\medskip
\textbf{(A)} \textit{The function $f$ belongs to $L^\infty_{\rm loc}([0,\infty),H)$.}
\medskip

Under assumption (A), for any $u^0\in H$, there exists a Leray-Hopf weak solution $u(t)$ of \eqref{fctnse} and \eqref{uzero}, see e.g. \cite{FMRTbook}. 
The large-time behavior of $u(t)$ is the focus of our study.
More specific conditions on $f$ will be imposed later. 

We note that, thanks to Remark 1(e) of  \cite{FRT2010}, the Leray-Hopf weak solutions in Definition \ref{lhdef} are the same as the weak solutions used in \cite[Chapter II, section 7]{FMRTbook}, even though they have slightly different formulations. Hence, according to inequality (A.39) in \cite[Chap. II]{FMRTbook}, we have for any Leray-Hopf weak solution $u(t)$ (in Definition \ref{lhdef}) that
\beq\label{iniener}
|u(t)|^2\le e^{-t}|u(0)|^2 + \int_0^t e^{-(t-\tau)} |f(\tau)|^2d\tau\quad \forall t>0.
\eeq

\subsection{Basic inequalities}
Below are some inequalities that will be needed in later estimates.
First, for any $\sigma,\alpha>0$, one has
\beq\label{mxe}
\max_{x\ge 0} (x^{\alpha}e^{-\sigma x})=d_0(\alpha,\sigma)\eqdef \Big(\frac{\alpha}{e\sigma }\Big)^{\alpha},
\eeq 
and, hence, 
\beq\label{mx2}
e^{-\sigma x}=e^{-\sigma (x+1)}e^\sigma\le d_0(\alpha,\sigma)e^\sigma (1+x)^{-\alpha}\quad \forall x\ge 0.
\eeq

Thanks to \eqref{mxe}, one can verify, for all $\alpha,\sigma>0$, that
\beq\label{als0}
|A^\alpha e^{-\sigma A}v| \le d_0(\alpha,\sigma) |v|\quad \forall v\in H,
\eeq
\beqs
|A^\alpha e^{-\sigma A^{1/2}}v| \le d_0(2\alpha,\sigma) |v|\quad \forall v\in H,
\eeqs 
and, consequently, 
\beq 
\label{als}
|A^\alpha v|=|(A^\alpha e^{-\sigma A^{1/2}})e^{\sigma A^{1/2}}v|\le d_0(2\alpha,\sigma) |e^{\sigma A^{1/2}}v|\quad  \forall v\in G_{0,\sigma}.
\eeq

For the bi-linear mapping $B(u,v)$, it follows from its boundedness that there exists a constant $K_*>0$ such that 
\beq\label{Bweak}
\|B(u,v)\|_{V'}\le  K_* \| u\|\, \|v\|\quad\forall u,v\in V.
\eeq

For stronger norms of $B(u,v)$, we recall from \cite[Lemma 2.1]{HM1} a convenient inequality. (See Foias-Temam paper  \cite{FT-Gevrey} for the original version.)  

\medskip
\textit{There exists a constant $K>1$ such that if $\sigma\ge 0$ and $\alpha\ge 1/2$, then 
\beq\label{AalphaB} 
|B(u,v)|_{\alpha ,\sigma }\le K^\alpha  |u|_{\alpha +1/2,\sigma } \, |v|_{\alpha +1/2,\sigma}\quad \forall u,v\in G_{\alpha+1/2,\sigma}.
\eeq
}

\begin{lemma}\label{ineqlem}
Let $\sigma,\lambda>0$.   One has, for all $t\ge0$, that
 \beq \label{i2}
 \int_0^t \frac{e^{-\sigma (t-\tau)}}{(1+\tau)^\lambda}d\tau
 \le \frac{d_1(\lambda,\sigma)}{(1+t)^\lambda},
\eeq
where
\beqs
d_1(\lambda,\sigma)\eqdef 2^\lambda (d_0(\lambda+1,\sigma)e^\sigma +\sigma^{-1})
= 2^\lambda\Big[\Big(\frac{\lambda+1}{e\sigma}\Big)^{\lambda+1}e^\sigma +\frac1\sigma\Big].
\eeqs
\end{lemma}
\begin{proof}
First, we have 
\begin{align*}
I&\eqdef   \int_0^ t\frac{e^{-\sigma(t- \tau)}}{(1+\tau)^\lambda} d\tau=  \int_0^{t/2}\frac{e^{-\sigma(t- \tau)}}{(1+\tau)^\lambda} d\tau + \int_{t/2}^ t\frac{e^{-\sigma(t- \tau)}}{(1+\tau)^\lambda} d\tau \\
&\le \int_0^{t/2} e^{-\sigma t/2} d\tau + \frac{1}{(1+t/2)^\lambda}\int_{t/2}^ t e^{-\sigma(t- \tau)}  d\tau
 \le \frac{t}{2}e^{-\sigma t/2} + \frac{1}{(1+t/2)^\lambda}\cdot \frac{1}{\sigma}.
\end{align*}

Note, by \eqref{mx2}, that 
$$
e^{-\sigma t/2}\le \frac{d_0(\lambda+1,\sigma)e^\sigma}{(1+t/2)^{\lambda+1}},
$$
which implies
\beqs
\frac{t}{2}e^{-\sigma t/2} \le \frac{C}{(1+t/2)^\lambda}, \text{ where }C=d_0(\lambda+1,\sigma)e^\sigma.
\eeqs

Thus, we obtain
\beqs
I \le \frac{1}{(1+t/2)^\lambda} \Big(C+\frac{1}{\sigma}\Big)\le   \frac{2^\lambda}{(1+t)^\lambda} \Big(C+\frac{1}{\sigma}\Big)
=\frac{d_1(\sigma,\lambda)}{(1+t)^\lambda},
\eeqs
which proves inequality \eqref{i2}.
 \end{proof}

\subsection{Large-time behavior of solutions of linearized NSE} 

We discuss the asymptotic behavior, in Sobolev-Gevrey spaces, of weak solutions of linearized NSE.

\begin{lemma}\label{odelem}
Let $\alpha,\sigma\ge 0$,  $\xi\in \mD(A^\alpha)$, and  $f$ be a function from $(0,\infty)$ to $G_{\alpha,\sigma}$ that satisfies 
  \beq\label{flem}
  |f(t)|_{\alpha,\sigma}\le M(1+t)^{-\lambda} \quad\text{a.e. in $(0,\infty)$ for some $M>0$.}\eeq 

  Suppose  
  \beq\label{wcond} w\in C([0,\infty),H_{\rm w})\cap L^1_{\rm loc}([0,\infty),V) , \text{ with }w'\in L^1_{\rm loc}([0,\infty),V'),
  \eeq 
  is a weak solution of 
 \beq\label{weq}
 w'=-Aw+\xi+f \text{ in $V'$ on $(0,\infty)$,}
  \eeq
  i.e., it holds, for all $v\in V$, that
  \beq\label{disweq}
 \ddt \inprod{w,v}=-\doubleinprod{w,v}+\inprod{\xi+f,v} \text{ in the distribution sense on $(0,\infty)$.}
  \eeq

Then, for any $\varep \in (0,1)$, there exists $C>0$ depending on $\varep$, $\lambda$, $M$, $|\xi|_{\alpha,\sigma}$ and $|w(0)|_{\alpha,\sigma}$ such that
\beq\label{yremain}
|w(t)-A^{-1}\xi|_{\alpha+1-\varep,\sigma}\le C(1+t)^{-\lambda} \quad \forall t\ge 1. 
\eeq
 \end{lemma}
 \begin{proof}
Let $N\in\sigma(A)$, and set $A_N=A|_{P_N H}$ which is an invertible linear map from $P_N H$ onto itself.
By taking $v\in P_NH$ in equation \eqref{disweq}, we deduce that $P_N w$ solves, in the $P_NH$-valued distribution sense on $(0,\infty)$, the equation
\beq\label{vy}
\ddt (P_Nw)=-A_N (P_Nw) +P_N\xi +P_Nf\quad \text{in  $P_N H$ on $(0,\infty)$.}
\eeq

Since $P_N H$ is a finite-dimensional Euclidean space, one has $P_Nw\in C([0,\infty),P_NH)$ and $(P_N w)'\in L^1_{\rm loc}([0,\infty),P_N H)$. Then the variation of constants formula still holds true for the solution $P_N w$ of \eqref{vy}. (See, for example, the arguments in \cite[Lemma 4.2]{HM2}.) 
By denoting $w_0=w(0)$, we have
\begin{align*}
  P_N w(t)&=e^{-tA_N}P_N w_0+\int_0^t e^{-(t-\tau)A_N}(P_N\xi+P_Nf(\tau))d\tau\\
  &=e^{-tA_N}P_N w_0+A_N^{-1}(P_N \xi-e^{-tA_N}P_N \xi)+\int_0^t e^{-(t-\tau)A_N}P_Nf(\tau)d\tau\\
  &=e^{-tA}P_N w_0+A^{-1}(P_N \xi-e^{-tA}P_N \xi)+\int_0^t e^{-(t-\tau)A}P_Nf(\tau)d\tau,
\end{align*}
  which yields
\beq\label{yform}
  P_N\big (w(t)-A^{-1}\xi\big )=e^{-tA}P_N w_0-A^{-1}e^{-tA}P_N \xi+\int_0^t e^{-(t-\tau)A}P_N f(\tau)d\tau.
\eeq
  
Let $\varep\in (0,1)$. 
Applying $A^{1-\varep}$ to both sides of \eqref{yform}, and estimating the $|\cdot|_{\alpha,\sigma}$ norm of the resulting quantities,   we obtain
\beq \label{ynorm}
  \begin{aligned}
  |P_N(w(t)-A^{-1}\xi)|_{\alpha+1-\varep,\sigma}
  &\le |A^{1-\varep}e^{-tA}w_0|_{\alpha,\sigma}+|A^{-\varep}e^{-tA}\xi|_{\alpha,\sigma}\\
  &\quad +\int_0^t |e^{-(t-\tau)A}A^{1-\varep} f(\tau)|_{\alpha,\sigma}d\tau.
    \end{aligned}
\eeq  

  We find bounds for each term on the right-hand side of the preceding inequality.
  
  \medskip
$\bullet$ Firstly, for $t\ge 1$, rewriting the first term on the right-hand side of \eqref{ynorm} and applying \eqref{als0} yield
  \begin{align*}
   |A^{1-\varep}e^{-tA}w_0|_{\alpha,\sigma}
   &= |A^{1-\varep}e^{-tA/2} e^{-tA/2} w_0|_{\alpha,\sigma}
   \le \Big[ \frac{(1-\varep)}{et/2}\Big]^{1-\varep}|e^{-tA/2} w_0|_{\alpha,\sigma}\\
   &\le \Big[ \frac{2(1-\varep)}{e}\Big]^{1-\varep} e^{-t/2} |w_0|_{\alpha,\sigma}. 
  \end{align*}
To compare $e^{-t/2}$ and $(1+t)^{-\lambda}$, we apply \eqref{mx2} to obtain
\beq\label{y1}
   |A^{1-\varep}e^{-tA}w_0|_{\alpha,\sigma}
 \le \Big[ \frac{2(1-\varep)}{e}\Big]^{1-\varep} \frac{d_0(\lambda,1/2)e^{1/2}}{(1+t)^{\lambda}}|w_0|_{\alpha,\sigma}.
\eeq

  \medskip
$\bullet$ 
 Secondly, the second term on the right-hand side of \eqref{ynorm} can be easily estimated by
  \beq\label{y2}
   |A^{-\varep}e^{-tA}\xi|_{\alpha,\sigma}
\le  | e^{-tA} \xi|_{\alpha,\sigma}
   \le e^{-t} |\xi|_{\alpha,\sigma}\le  \frac{d_0(\lambda,1)e}{(1+t)^{\lambda}}|\xi|_{\alpha,\sigma}. 
\eeq

  \medskip
$\bullet$ Thirdly, dealing with  the last integral in \eqref{ynorm}, we split it into two integrals
\beq\label{y3}
\int_0^t |e^{-(t-\tau)A}A^{1-\varep}f(\tau)|_{\alpha,\sigma}d\tau=I_1+I_2,
\eeq
where 
\beqs
I_1= \int_0^{t/2} |e^{-(t-\tau)A}A^{1-\varep}f(\tau)|_{\alpha,\sigma}d\tau,\quad 
I_2= \int_{t/2}^t |e^{-(t-\tau)A}A^{1-\varep}f(\tau)|_{\alpha,\sigma}d\tau.
\eeqs
  
  \medskip
For $I_1$, we have for $t\ge 1$ that
\begin{align*}
 I_1&= \int_0^{t/2} \Big|e^{-(t-\tau)A/2}\Big( e^{-(t-\tau)A/2} A^{1-\varep} f(\tau)\Big)\Big|_{\alpha,\sigma}d\tau\\
 &\le \int_0^{t/2} \Big|e^{-(t/4) A}A^{1-\varep} e^{-(t-\tau)A/2}f(\tau)\Big|_{\alpha,\sigma}d\tau.
\end{align*}

Utilizing \eqref{als0} and then using hypothesis \eqref{flem}, we obtain
\begin{align*}
 I_1&\le \int_0^{t/2} \Big [\frac{1-\varep}{et/4}\Big ]^{1-\varep}|e^{-(t-\tau)A/2}f(\tau)|_{\alpha,\sigma}d\tau\\
  &\le \Big [\frac{1-\varep}{et/4}\Big ]^{1-\varep} \int_0^{t/2}  e^{-(t-\tau)/2} M(1+\tau)^{-\lambda}d\tau\\
    &= M\Big [\frac{1-\varep}{et/4}\Big ]^{1-\varep} e^{\frac{-t}{4}}\int_0^{t/2}  e^{-(t/2-\tau)/2} (1+\tau)^{-\lambda}d\tau.
\end{align*}

Then by Lemma \ref{ineqlem}
\begin{align*}
 I_1
 &\le  M \Big [\frac{4(1-\varep)}{e t}\Big ]^{1-\varep} e^{-t/4} \frac{d_1(\lambda,1/2)}{(1+t/2)^{\lambda}}
 \le  M \Big [\frac{4(1-\varep)}{e }\Big ]^{1-\varep} \frac{e^{-t/4}}{t^{1-\varep}} \frac{2^\lambda d_1(\lambda,1/2)}{(1+t)^{\lambda}}.\\
\end{align*}
Thus, for $t\ge 1$
\beq\label{y31}
I_1\le   M \Big [\frac{4(1-\varep)}{e }\Big ]^{1-\varep} e^{-1/4} 2^\lambda d_1(\lambda,1/2) \frac{1}{(1+t)^{\lambda}}.
\eeq

\medskip

For $I_2$, we apply \eqref{als0} and use \eqref{flem} to estimates its integrand, for $t/2<\tau<t$, by
\begin{align*}
&|e^{-(t-\tau)A}A^{1-\varep}f(\tau)|_{\alpha,\sigma}
\le  e^{-(t-\tau)/2} |e^{-(t-\tau)A/2}A^{1-\varep}f(\tau)|_{\alpha,\sigma}\\
&\quad \le  e^{-(t-\tau)/2} \Big[\frac{1-\varep}{e(t-\tau)/2}\Big]^{1-\varep} |f(\tau)|_{\alpha,\sigma} 
\le  \Big[\frac{2(1-\varep)}{e}\Big]^{1-\varep} \frac{e^{-(t-\tau)/2}}{(t-\tau)^{1-\varep}}\cdot  \frac{M}{(1+\tau)^\lambda} \\
&\quad \le  \Big[\frac{2(1-\varep)}{e}\Big]^{1-\varep}\frac{M}{(1+t/2)^\lambda} \cdot \frac{e^{-(t-\tau)/2}}{(t-\tau)^{1-\varep}}.
\end{align*}
Hence,
\begin{align*}
I_2
&\le  \Big[\frac{2(1-\varep)}{e}\Big]^{1-\varep}\frac{M}{(1+t/2)^\lambda} \int_{t/2}^t \frac{e^{-(t-\tau)/2}}{(t-\tau)^{1-\varep}}d\tau
\le \Big[\frac{2(1-\varep)}{e}\Big]^{1-\varep}\frac{2^\lambda M}{(1+t)^\lambda} \int_0^{t/2} \frac{e^{-z/2}}{z^{1-\varep}}dz.
\end{align*}
We estimate the last integral by
\begin{align*}
 \int_0^{t/2} \frac{e^{-z/2}}{z^{1-\varep}}dz
 &=\int_0^{1/2} \frac{e^{-z/2}}{z^{1-\varep}}dz+\int_{1/2}^{t/2} \frac{e^{-z/2}}{z^{1-\varep}}dz
\le \int_0^{1/2} \frac{1}{z^{1-\varep}}dz+ 2^{1-\varep}\int_{1/2}^{t/2} e^{-z/2}dz\\
&\le 2^{-\varep}\varep^{-1}+ 2^{2-\varep} e^{-1/4} =2^{-\varep}(\varep^{-1}+ 4 e^{-1/4}) .
\end{align*}
Therefore,
\beq\label{y32}
I_2\le M\Big[\frac{1-\varep}{e}\Big]^{1-\varep}2^{1-2\varep+\lambda} ( \varep^{-1}+ 4 e^{-1/4} )\cdot \frac{1}{(1+t)^{\lambda}}. 
\eeq

Combining \eqref{ynorm}--\eqref{y32}, we obtain 
\beq\label{yNrem}
|P_N\big( w(t)-A^{-1}\xi\big)|_{\alpha+1-\varep,\sigma}\le C(1+t)^{-\lambda} \quad \forall t\ge 1,
\eeq
with constant $C$ independent of $N$.
Since  $A^{-1}\xi$ belongs to $G_{\alpha+1-\varep,\sigma}$, this bound shows that $w(t)$ also belongs to $G_{\alpha+1-\varep,\sigma}$.
By passing $N\to\infty$ in \eqref{yNrem}, we obtain \eqref{yremain}.
The proof is complete.
\end{proof}

The particular case $\xi=0$ has a special meaning, and we state the result separately here.

\begin{lemma}\label{linearNSE}
Let $\alpha,\sigma\ge 0$,  and  suppose $f$ is a function in $L^\infty_{\rm loc}([0,\infty),G_{\alpha,\sigma})$. 
Let $w$ satisfy \eqref{wcond} and be a weak solution of 
 \beqs
 w'=-Aw+f \text{ in $V'$ on $(0,\infty)$.}
  \eeqs
  
  \begin{enumerate}[label=\rm (\roman*)]
   \item Then $w(t)\in G_{\alpha+1-\varep,\sigma}$ for all $\varep\in(0,1)$ and  $t>0$.

\item If, in addition, $f$ satisfies \eqref{flem}, then, for any $\varep \in (0,1)$, there exists $C>0$ depending on $\varep$, $\lambda$, $M$ and $|w(0)|_{\alpha,\sigma}$ such that
\beq\label{wremain}
|w(t)|_{\alpha+1-\varep,\sigma}\le C(1+t)^{-\lambda} \quad \forall t\ge 1. 
\eeq
  \end{enumerate}
 \end{lemma}
\begin{proof}
We set $\xi=0$ in \eqref{weq} and follow the proof of Lemma \ref{odelem}. In this case, \eqref{ynorm} reads, for all $t>0$, as
\beq \label{wnorm}
  |P_Nw(t)|_{\alpha+1-\varep,\sigma}
\le |A^{1-\varep}e^{-tA}w_0|_{\alpha,\sigma}
 +\int_0^t |e^{-(t-\tau)A}A^{1-\varep} f(\tau)|_{\alpha,\sigma}d\tau.
\eeq

(i) Let $T>0$. There is $M_0>0$ such that $|f(t)|_{\alpha,\sigma}\le M_0$ a.e. in $(0,T)$.
For $t\in(0,T)$, we use \eqref{als0} to estimate 
\beqs
|A^{1-\varep}e^{-tA}w_0|_{\alpha,\sigma}\le \Big[\frac{1-\varep}{et}\Big]^{1-\varep} |w_0|_{\alpha,\sigma},
\eeqs
\begin{align*}
& \int_0^t |e^{-(t-\tau)A}A^{1-\varep} f(\tau)|_{\alpha,\sigma}d\tau
\le \Big[\frac{1-\varep}{e}\Big]^{1-\varep} \int_0^t \frac{|f(\tau)|_{\alpha,\sigma}}{(t-\tau)^{1-\varep}}d\tau\\
&\quad \le \Big[\frac{1-\varep}{e}\Big]^{1-\varep}  \int_0^t \frac{M_0}{(t-\tau)^{1-\varep}}d\tau 
= \Big[\frac{1-\varep}{e}\Big]^{1-\varep} M_0 \varep^{-1}t^\varep .
\end{align*}
Utilizing these estimates, we can pass $N\to\infty$ in \eqref{wnorm}, and obtain 
$$|w(t)|_{\alpha+1-\varep,\sigma}
\le \Big[\frac{1-\varep}{e}\Big]^{1-\varep}\Big( t^{\varep-1} |w_0|_{\alpha,\sigma}+M_0 \varep^{-1}t^\varep\Big),$$ 
thus, $w(t)\in G_{\alpha+1-\varep,\sigma}$.

(ii) This part is the same as Lemma \ref{odelem}, and \eqref{wremain} follows \eqref{yremain}.
\end{proof}

\section{Asymptotic estimates for the Leray-Hopf weak solutions}\label{Gdecay}

The goal of this section is to establish the power-decay for any Leray-Hopf weak solutions whenever the force is power-decaying.
The first theorem concerns the Gevrey estimates for the solutions for positive time when the initial data is small in a Sobolev norm, and the force is small in a Gevrey norm.

\begin{theorem} \label{theo22}
Let $\lambda>0$, $\sigma\ge 0$, and $\alpha\ge1/2$ be given numbers.
Suppose
\begin{align}\label{usmall}
|A^\alpha u^0|&\le c_0,\\
\label{fta}
|f(t)|_{\alpha-1/2,\sigma}&\le c_1 (1+t)^{-\lambda}\quad\text{a.e. in } (0,\infty),
\end{align} 
where
\beq\label{c01}
c_0=c_0(\alpha,\lambda)\eqdef \frac{c_*}{\max\{1,\sqrt{M_1}\}}
\quad \text{and}\quad 
c_1= c_1(\alpha,\lambda)\eqdef \frac{c_*}{\sqrt{3  M_2}},
\eeq 
with 
\beqs 
c_*=c_{*,\alpha}\eqdef \frac1{12K^\alpha},
\quad 
M_1=M_{1,\lambda}\eqdef d_0(2\lambda,1)e,
\quad M_2=M_{2,\lambda}\eqdef d_1(2\lambda,1).
\eeqs

Then there exists a unique global strong solution $u(t)$ of \eqref{fctnse} and \eqref{uzero}, which, furthermore, satisfies 
$u\in C([0,\infty),\mathcal D(A^\alpha))$ and 
\begin{align}\label{uest}
|u(t)|_{\alpha,\sigma}&\le \sqrt{2}c_* (1+t)^{-\lambda}\quad \forall t\ge t_*,\\
\label{intAa}
 \int_t^{t+1} |u(\tau)|_{\alpha+1/2,\sigma}^2d\tau
 &\le  2c_*^2 \Big (1 +\frac{1}{ 2M_2}\Big) (1+t)^{-2\lambda}\quad \forall t\ge t_*,
 \end{align}
 where $t_*=12\sigma$.
\end{theorem}
\begin{proof}
We will perform formal calculations below. They can be made rigorous by applying to solutions of the Galerkin approximations and then pass to the limit.

(a) \textit{Case $\sigma>0$.} We denote by $\varphi$ a $C^\infty$-function on $\R$ that satisfies $\varphi(t)=0$ on $(-\infty,0]$, $\varphi(t)=\sigma$ on $[t_*,\infty)$, and $0< \varphi'(t)< 2\sigma/t_*=1/6$ on $(0,t_*)$.
We derive from \eqref{fctnse} that 
\begin{align}
\ddt (A^{\alpha}e^{\varphi(t) A^{1/2}}u) 
&=\varphi'(t)A^{1/2}A^{\alpha}e^{\varphi(t) A^{1/2}} u+A^{\alpha}e^{\varphi(t) A^{1/2}}\frac{du}{dt}  \notag\\
&=\varphi'(t)A^{\alpha+1/2}e^{\varphi(t) A^{1/2}} u+A^{\alpha}e^{\varphi(t) A^{1/2}}(-Au-B(u,u)+f).\label{daeu}
\end{align}

By taking the inner product in $H$ of  \eqref{daeu} with $A^{\alpha}e^{\varphi(t) A^{1/2}}u(t)$, we obtain 
\begin{align*}
&\frac12\ddt |u|_{\alpha,\varphi(t)}^2   + |A^{1/2}u|_{\alpha,\varphi(t)}^2
=\varphi'(t)\langle A^{\alpha+1/2}e^{\varphi(t) A^{1/2}}u,A^{\alpha}e^{\varphi(t) A^{1/2}} u\rangle\\ 
&\quad -\langle A^{\alpha}e^{\varphi(t) A^{1/2}}B(u,u),A^{\alpha}e^{\varphi(t) A^{1/2}}u\rangle+ \langle A^{\alpha-1/2}e^{\varphi(t) A^{1/2}}f,A^{\alpha+1/2}e^{\varphi(t) A^{1/2}}u\rangle.
\end{align*}

Using the Cauchy-Schwarz inequality, and estimating the second term on the right-hand side  by \eqref{AalphaB}, we get
\beq \label{s0}
\begin{aligned}
\frac12\ddt |u|_{\alpha,\varphi(t)}^2   + |A^{1/2}u|_{\alpha,\varphi(t)}^2
 &\le \varphi'(t) |A^{1/2}u|_{\alpha,\varphi(t)}^2\\
 &\quad +K^\alpha |A^{1/2}u|_{\alpha,\varphi(t)}^2 |u|_{\alpha,\varphi(t)}
+ |f(t)|_{\alpha-1/2,\varphi(t)}|A^{1/2}u|_{\alpha,\varphi(t)}.
\end{aligned}
\eeq 

Using the bound of $\varphi'(t)$ and applying Cauchy's inequality to the last term gives
\begin{align*}
&\frac12\ddt |u|_{\alpha,\varphi(t)}^2   + |A^{1/2}u|_{\alpha,\varphi(t)}^2
\le \frac16 |A^{1/2}u|_{\alpha,\varphi(t)}^2 \\
&+K^\alpha |u|_{\alpha,\varphi(t)} |A^{1/2}u|_{\alpha,\varphi(t)}^2 
 +\frac16|A^{1/2}u|_{\alpha,\varphi(t)}^2+ \frac3{2}|f(t)|_{\alpha-1/2,\varphi(t)}^2,
\end{align*}
which, together with the fact $\varphi(t)\le\sigma$, implies
\beq\label{s1}
\frac12\ddt |u|_{\alpha,\varphi(t)}^2 + \Big(1-\frac13 -K^\alpha |u|_{\alpha,\varphi(t)}\Big)|A^{1/2}u|_{\alpha,\varphi(t)}^2 \le \frac3{2}|f(t)|_{\alpha-1/2,\sigma}^2.
\eeq

(b) \textit{Case $\sigma=0$.} Let $\varphi(t)=0$ for all $t\in\R$. Then the first term in \eqref{s0} vanishes. Applying Cauchy's inequality to the last term of \eqref{s0}, we obtain
 \beq\label{dtAa}
 \frac12\ddt |A^\alpha u|^2+\Big(1-\frac13-K^\alpha|A^\alpha u|\Big)|A^{\alpha+1/2}u|^2\le \frac34|A^{\alpha-1/2}f|^2\le \frac32|A^{\alpha-1/2}f|^2.
 \eeq
Hence, we have the same inequality as \eqref{s1}.

\medskip
(c) For both cases $\sigma>0$ and $\sigma=0$, let $T\in(0,\infty)$. Note that 
$$|u(0)|_{\alpha,\varphi(0)}=|A^\alpha u^0|<2c_0\le 2c_*.$$ 

Assume that
\beq\label{uT}
|u(t)|_{\alpha,\varphi(t)}\le 2c_*\quad \forall t\in[0,T).
\eeq

This and the definition of $c_*$ give
\beq\label{Ku}
K^\alpha |u(t)|_{\alpha,\varphi(t)} \le 2c_* K^\alpha=1/6 \quad \forall t\in[0,T).
\eeq

For $t\in (0,T)$, we have from \eqref{s1}, \eqref{dtAa}, and \eqref{Ku} that 
\beq\label{s2}
\ddt |u|_{\alpha,\varphi(t)}^2 + |A^{1/2}u|_{\alpha,\varphi(t)}^2 \le 3|f(t)|_{\alpha-1/2,\sigma}^2.
\eeq

Applying Gronwall's inequality in \eqref{s2} yields for all $t\in(0,T)$ that
\begin{align*}
|u(t)|_{\alpha,\varphi(t)}^2
&\le e^{-t}|u^0|_{\alpha,0}^2+3\int_0^t e^{-(t-\tau)}|f(\tau)|_{\alpha-1/2,\sigma}^2 d\tau\\
\text{(by \eqref{usmall} and \eqref{fta}) }&\le e^{-t}c_0^2+3c_1^2\int_0^t \frac{e^{-(t-\tau)}}{(1+\tau)^{2\lambda}} d\tau.
\end{align*}

Using \eqref{mx2} to compare $e^{-t}$ with $(1+t)^{-2\lambda}$, and estimating the last integral by \eqref{i2} yield
\beqs 
|u(t)|_{\alpha,\varphi(t)}^2 \le   \frac{M_1 c_0^2}{(1+t)^{2\lambda}} + \frac{3c_1^2 M_2}{(1+t)^{2\lambda}}
\le   \frac{2 c_*^2}{(1+t)^{2\lambda}}.
\eeqs
This implies 
\beq\label{s4}
|u(t)|_{\alpha,\varphi(t)}
\le  \sqrt{2} c_* (1+t)^{-\lambda} \quad\forall t\in[0,T).
\eeq

Letting $t\to T^-$ in \eqref{s4} gives
\beq\label{limTu}
\lim_{t\to T^-}|u(t)|_{\alpha,\varphi(t)}
\le \sqrt{2}c_* (1+T)^{-\lambda} < 2c_*.
\eeq

Comparing \eqref{limTu} to \eqref{uT}, and by the standard contradiction argument, we deduce that the inequalities \eqref{uT} and \eqref{s4} hold for $T=\infty$.  
Then, thanks to $\varphi(t)=\sigma$ for all $t\ge t_*$,  inequality \eqref{s4} implies \eqref{uest}.

(d) For $t\ge t_*$, by integrating \eqref{s2} from $t$ to $t+1$, and using estimates \eqref{uest}, \eqref{fta}, we obtain
\begin{align*}
\int_t^{t+1} |A^{1/2}u(\tau)|_{\alpha,\sigma}^2d\tau
& \le | u(t)|_{\alpha,\sigma}^2+3c_1^2\int_t^{t+1}(1+\tau)^{-2\lambda}d\tau\\
&\le 2c_*^2 (1+t)^{-2\lambda }+3c_1^2(1+t)^{-2\lambda }\\
&= \Big(2 c_*^2 +\frac{c_*^2}{ M_2}\Big)(1+t)^{-2\lambda}.
\end{align*}
Then inequality \eqref{intAa} follows. The proof is complete.
\end{proof}

In the next theorem, we establish the power decay, as $t\to\infty$, for any Leray-Hopf weak solutions. Its proof combines the energy inequalities \eqref{Lenergy} and \eqref{iniener} with successive use of Theorem \ref{theo22}.

\begin{theorem}\label{theo23}
Assume  that there are $\sigma\ge 0$, $\alpha\ge 1/2$ and $\mu_1>0$ such that
\beq\label{falphaonly}
|f(t)|_{\alpha,\sigma}=\mathcal O(t^{-\mu_1})\quad\text{as }t\to\infty.
\eeq

Let $u(t)$ be a Leray-Hopf weak solution of \eqref{fctnse}.  
Then there exists  $T_*>0$ 
such that $u(t)$ is a regular solution of \eqref{fctnse} on $[T_*,\infty)$, and 
 for any $\varep\in(0,1)$, there exists $C>0$ such that 
 \beq\label{uep}
 |u(T_*+t)|_{\alpha+1-\varep,\sigma} \le C(1+ t)^{-\mu_1},
 \eeq
 \beq\label{Buep}
|B(u(T_*+t),u(T_*+t))|_{\alpha+1/2-\varep,\sigma}\le C (1+t)^{-2\mu_1},
\eeq
for all $t\ge 0$.
\end{theorem}
\begin{proof}
The proof is divided into two parts.

\medskip
\noindent\textbf{Part A.} We prove the following weaker version of the statements.

\textit{For any $\lambda\in(0,\mu_1)$, there exists  $T_*>0$ 
such that $u(t)$ is a regular solution of \eqref{fctnse} on $[T_*,\infty)$, and one has for all $t\ge 0$ that
 \beq\label{preus0}
 |u(T_*+t)|_{\alpha+1/2,\sigma} \le K^{-\alpha-1/2}(1+ t)^{-\lambda},
 \eeq
 \beq\label{preBlt}
|B(u(T_*+t),u(T_*+t))|_{\alpha,\sigma}\le K^{-\alpha-1}(1+t)^{-2\lambda},
\eeq
where $K$ is the constant in inequality \eqref{AalphaB}.}

The proof of this part consists several steps.

\medskip
\textit{Step 1.} By assumption (A) and \eqref{falphaonly}, there exists  $M>0$ such that
\beq\label{fkappa}
|f(t)|\le M (1+t)^{-\mu_1} \quad\text{a.e. in } (0,\infty).
\eeq

It follows \eqref{iniener} and \eqref{fkappa} that, for all $t>0$,
\begin{align*} \label{uener2}
|u(t)|^2
&\le e^{-t}|u_0|^2 + M^2 \int_0^t \frac{e^{-(t-\tau)}}{(1+\tau)^{2\mu_1}}d\tau \\
\text{(by \eqref{mx2} and \eqref{i2})}
&\le C_1(1+t)^{-2\mu_1}|u_0|^2  + M^2 C_2 (1+t)^{-2\mu_1},
\end{align*}
where 
$C_1=d_0(2\mu_1,1)e$ and $C_2=d_1(2\mu_1,1)$.
Thus, 
\beq\label{uenerM}
|u(t)|^2\le  (|u_0|^2 C_1 +M^2 C_2 )(1+t)^{-2\mu_1}\quad\forall t\ge 0.
\eeq

In \eqref{Lenergy}, we estimate
\beqs
|\inprod{f(\tau),u(\tau)}|\le \frac12 |u(\tau)|^2+\frac12 |f(\tau)|^2 \le \frac12 \|u(\tau)\|^2+\frac12 |f(\tau)|^2.
\eeqs
Hence, we obtain 
	\beq\label{intt0}
		|u(t)|^2+\int_{t_0}^t \|u(\tau)\|^2d\tau\le |u(t_0)|^2+\int_{t_0}^t |f(\tau)|^2\ d\tau,
	\eeq
for $t_0=0$ and almost all $t_0\in(0,\infty)$, and all $t\ge t_0$.

Letting $t=t_0+1$ in \eqref{intt0}, using \eqref{uenerM} to estimate $|u(t_0)|^2$, and \eqref{fkappa} to estimate $|f(\tau)|$, we derive
\beq\label{uVest}
\begin{aligned}
\int_{t_0}^{t_0+1}\|u(\tau)\|^2d\tau 
&\le ( |u_0|^2 C_1 +M^2 C_2 )(1+t_0)^{-2\mu_1} + M^2 (1+t_0)^{-2\mu_1} \\
&= (|u_0|^2 C_1 +M^2 C_2 +M^2) (1+t_0)^{-2\mu_1}.
\end{aligned}
\eeq 

To establish \eqref{uVest} for any $t_0$, we use the following approximation.
Let $t\ge 0$ be arbitrary. There exists a sequence $\{t_n\}_{n=1}^\infty$ in $(0,\infty)$ such that  $\lim_{n\to\infty}t_n=t$ and  \eqref{uVest} holds for $t_0=t_n$, i.e., 
\beqs
\int_{t_n}^{t_n+1}\|u(\tau)\|^2d\tau 
\le  \Big(|u_0|^2 C_1 +M^2 C_2 +M^2 \Big) (1+t_n)^{-2\mu_1}.
\eeqs
Then letting $n\to\infty$ gives
 \beq\label{tt1}
\int_t^{t+1}\|u(\tau)\|^2d\tau 
\le \Big(|u_0|^2 C_1 +M^2 C_2 +M^2 \Big) (1+t)^{-2\mu_1} \quad\forall t\ge 0.
\eeq

\medskip
\textit{Step 2.} We prove that  there exists   $T>0$ so that
\begin{align}\label{Aut}
|A^{\alpha+1/2} u(T)|&\le c_0(\alpha+1/2,\lambda),\\
\label{ftt}
|f(T+t)|_{\alpha,\sigma}&\le c_1(\alpha+1/2,\lambda) (1+t)^{-\lambda} \quad \forall t\ge 0.
\end{align}

(a) \textit{Case $\sigma>0$.}
Set $\lambda'=(\lambda + \mu_1)/2$, which is a number in the interval $(\lambda, \mu_1)$.
Thanks to the decay in \eqref{tt1} and \eqref{falphaonly}, there exists $t_0>0$ such that
\begin{align*}
|A^{1/2}u(t_0)|&<c_0(1/2,\lambda'),\\
|f(t_0+t)|_{0,\sigma}&\le c_1(1/2,\lambda')(1+t)^{-\lambda'}\quad \forall t\ge 0,
\end{align*}
where $c_0(\cdot,\cdot)$ and $c_1(\cdot,\cdot)$ are defined in \eqref{c01}.

Applying Theorem \ref{theo22} to solution $t\mapsto u(t_0+t)$, force  $t\mapsto f(t_0+t)$ with parameters $\alpha=1/2$ and $\lambda= \lambda'$, we obtain from \eqref{uest} that 
\beq\label{ut1}
|u(t_0+t)|_{1/2,\sigma}\le \sqrt 2 c_{*,1/2} (1+t)^{-\lambda'}\le K^{-1/2} (1+t)^{-\lambda'}\quad \forall t\ge t_*\eqdef 12\sigma.
\eeq

Then by \eqref{als}, we have for all $t\ge t_*$ that
\beqs
|A^{\alpha+1/2} u(t_0+t)|
\le  d_0(2\alpha+1,\sigma) |e^{\sigma A^{1/2}} u(t_0+t)|
\le d_0(2\alpha+1,\sigma)   |u(t_0+t)|_{1/2,\sigma},
\eeqs
and, thanks to \eqref{ut1},
\beq\label{Aaut0}
|A^{\alpha+1/2} u(t_0+t)|
\le d_0(2\alpha+1,\sigma)  K^{-1/2} (1+t)^{-\lambda'}.
\eeq

Since $\lambda'>\lambda$, it follows \eqref{Aaut0} and \eqref{falphaonly} that there is a sufficiently large  $T>t_0+t_*$ so that \eqref{Aut} and \eqref{ftt} hold.

\medskip
(b) \textit{Case $\sigma=0$.}
First we claim that 

\medskip
\textit{Claim:} If $j\in \N$ such that  $j\le 2\alpha+1$ and 
\beq\label{intAj}
\lim_{t\to\infty}\int_t^{t+1}|A^{j/2}u(\tau)|^2d\tau=0, 
\eeq
then
\beq\label{intAp1}
\lim_{t\to\infty}\int_t^{t+1}|A^{(j+1)/2}u(\tau)|^2d\tau=0.
\eeq

\textit{Proof of Claim.} Note that $(j-1)/2 \le\alpha$, and thanks to \eqref{falphaonly}, we have 
\beq\label{fj}
|A^\frac{j-1}2f(t)|=\mathcal O(t^{-\mu_1 })\quad\text{as }t\to\infty.
\eeq

By \eqref{intAj} and \eqref{fj},  there is $T>0$ so that
\beqs
|A^{j/2} u(T)|\le c_0(j/2,\lambda),
\eeqs
\beqs
|A^{j/2-1/2}f(T+t)|\le c_1(j/2,\lambda)(1+t)^{-\lambda} \quad \forall t\ge 0.
\eeqs

Applying Theorem \ref{theo22}  to $u(T+\cdot)$, $f(T+\cdot)$,  $\alpha:=j/2$, $\sigma:=0$,  we obtain 
\beqs
\int_t^{t+1}|A^{(j+1)/2}u(\tau)|^2d\tau=\mathcal O(t^{-2\lambda})\quad\text{as }t\to\infty,
\eeqs
which proves \eqref{intAp1}. 

\medskip
Now, let $m$ be a non-negative integer such that
$2\alpha \le m<2\alpha+1$. 

Note that $m\ge 1$, and, because of \eqref{tt1}, condition \eqref{intAj} holds true for $j=1$.
Hence we obtain \eqref{intAp1} with $j=1$, which is \eqref{intAj} for $j=2$.
This way, we are able to apply the \textit{Claim} recursively for $j=1,2,\ldots,m$, and obtain, when $j=m$, from \eqref{intAp1} that
\beqs
\lim_{t\to\infty} \int_t^{t+1}|A^{(m+1)/2}u(\tau)|^2d\tau=0.
\eeqs
Since $\alpha\le m/2$, it follows that
\beq\label{intAm}
\lim_{t\to\infty}\int_t^{t+1}|A^{\alpha+1/2}u(\tau)|^2d\tau=0.
\eeq

By \eqref{intAm} and \eqref{falphaonly}, we assert that there is $T>0$ so that \eqref{Aut} and \eqref{ftt} hold.

\medskip
\textit{Step 3.}
With $T>0$ in \text{Step 2}, we apply Theorem \ref{theo22} to $u(T+\cdot)$, $f(T+\cdot)$, $\alpha:=\alpha+1/2$, and obtain that there is $T_*>T+t_*$  such that
\beqs
|u(T_*+t)|_{\alpha+1/2,\sigma}\le \sqrt 2 c_{*,\alpha+1/2}  (1+t)^{-\lambda}
\le \frac{1}{K^{\alpha+1/2}}(1+t)^{-\lambda}\quad \forall t\ge 0.
\eeqs 
This proves \eqref{preus0}.
Then  applying inequality \eqref{AalphaB} with the use of estimate \eqref{preus0} yields \eqref{preBlt}. 

\medskip
\noindent\textbf{Part B.} We now prove \eqref{uep} and \eqref{Buep}. 
We write equation \eqref{fctnse} as
\beq\label{ulin} u_t+Au=F(t)\eqdef -B(u(t),u(t))+f(t) .
\eeq

By Part A, we set $\lambda= \mu_1/2$ in \eqref{preBlt}  to obtain

\beq
|B(u(t),u(t))|_{\alpha,\sigma}=\mathcal O(t^{-\mu_1}) \text{ as }t\to\infty.
\eeq

From this and \eqref{falphaonly}, we have
\beqs
|F(t)|_{\alpha,\sigma}=\mathcal O(t^{-\mu_1})\text{ as }t\to\infty.
\eeqs

Applying part (ii) of Lemma \ref{linearNSE} to \eqref{ulin} on $(T,\infty)$ for some sufficiently large $T$ and any $\varep\in(0,1)$, we obtain the first inequality \eqref{uep}. Then the second inequality \eqref{Buep} follows \eqref{AalphaB} and \eqref{uep}. 
(In these arguments, the values of  $T_*$ and $C$ were adjusted appropriately.) The proof is complete.
\end{proof}

\begin{remark}
 The estimates \eqref{preus0} and \eqref{preBlt} are similar to those in \cite[Proposition 3.4]{HM2}.
 The improved estimates \eqref{uep} and \eqref{Buep} with the stronger norms come from the better regularity result in Lemma \ref{linearNSE}.
\end{remark}

\section{Asymptotic expansions}\label{expsec}

This section consists of the first main results of this paper. Briefly speaking, when the force has a finite or infinite expansion in terms of power-decaying functions, then any Leray-Hopf weak solution will have an asymptotic expansion of the same type.

\subsection{Finite expansions}\label{finsubsec}

We start with the following consideration for the force $f(t)$.

\textbf{(B1)} \textit{Suppose there exist numbers $\sigma\ge 0$, $\alpha\ge 1/2$,  an integer $N_0\ge1$, strictly increasing, positive  numbers $\gamma_n$ and functions  $\psi_n\in G_{\alpha,\sigma}$ for $1\le n\le N_0$, and a  number $\delta>0$ such that  
\beq \label{fgam}
\Big|f(t)-\sum_{n=1}^{N_0} \psi_n t^{-\gamma_n}\Big|_{\alpha,\sigma}=\mathcal O(t^{-\gamma_{N_0}-\delta}) \quad\text{as }t\to\infty.
\eeq
}


Note from \eqref{fgam} that  $f(t)$ belongs to $G_{\alpha,\sigma}$ for all $t$ sufficiently large.

Although the force $f(t)$ has an expansion in terms of $t^{-\gamma_n}$'s, the solution $u(t)$ of the NSE may not.
In fact, due to the system's nonlinearity and time derivative, $u(t)$ may be expanded in terms of functions of different powers, which we describe now. 

\medskip
We define the following set of powers generated by $\gamma_n$'s and 1:
\begin{align*}
S_*=\Big \{  (\sum_{j=1}^p \gamma_{n_j}) + k:&\text{ for some } p\ge 1,\ 1\le n_1,n_2,\ldots,n_p\le N_0, \\
&\text{ and some integer }k\ge 0\Big  \}. 
\end{align*}

Note that $S_*$ is an infinite subset of $(0,\infty)$, and possesses the property
\beq \label{Sprop}
\forall x,y\in S_*:\ x+1,x+y\in S_*.
\eeq 

By ordering this set, one has
\beq \label{museq}
S_*=\{\mu_n:n\in\N\}, \text{ where $\mu_n$'s are strictly increasing.}
\eeq

The set of powers that will be used for the expansion of $u(t)$ is 
$$S=S_*\cap [\gamma_1,\gamma_{N_0}].$$ 

This set $S$ is finite, and
\beq 
S=\{ \mu_n: 1 \le n\le N_*\}\text{ for some }N_*\ge N_0.
\eeq

Note that $\mu_1=\gamma_1$ and $\mu_{N_*}= \gamma_{N_0}$. 
Then we rewrite \eqref{fgam} as
\beq \label{ffinite}
\Big|f(t)-\sum_{n=1}^{N_*} \phi_n t^{-\mu_n}\Big|_{\alpha,\sigma}=\mathcal O(t^{-\mu_{N_*}-\delta}) \quad\text{as }t\to\infty,
\eeq
where $\phi_n\in G_{\alpha,\sigma}$ for $1\le n\le N_*$, which can be defined explicitly as follows.
If there exists $k\in [1,N_0]$ such that $\mu_n=\gamma_k$, then, with such $k$, $\phi_n=\psi_k$; otherwise, $\phi_n=0$.

Our first result on the expansion of Leray-Hopf weak solutions is the following.

\begin{theorem}
\label{finitetheo}
Assume {\rm (B1)}. Let $\mu_n$'s be as in \eqref{museq}, and let the corresponding equation \eqref{ffinite} hold true. 
Define $\xi_1$, $\xi_2$, $\ldots$, $\xi_{N_*}$ recursively by  
\begin{align} \label{xi1}
\xi_1&=A^{-1}\phi_1,\\
\label{xin1}
\xi_n&=A^{-1}\Big(\phi_n + \chi_n - \sum_{\stackrel{1\le k,m\le n-1,}{\mu_k+\mu_m=\mu_n}}B(\xi_k,\xi_m)\Big) \quad\text{for } 2\le n\le N_*,
\end{align}
where 
\beq \label{chin}
\chi_n=
\begin{cases}
\mu_p \xi_p, & \text{ if there exists an integer  $p\in [1, n-1]$ such that $\mu_p +1= \mu_n$},\\
0,&\text{ otherwise}.
\end{cases}
\eeq 

Let $u(t)$ be a Leray-Hopf weak solution of \eqref{fctnse} and \eqref{uzero}. 
Then it holds for all $\rho \in (0,1)$ that
\beq \label{ufinite}
\Big |u(t)-\sum_{n=1}^{N_*} \xi_n t^{-\mu_n}\Big|_{\alpha+1-\rho,\sigma}=\mathcal O(t^{-\mu_{N_*}-\varep_*})\quad\text{as }t\to\infty,
 \eeq
where
\beq\label{epstar}
\varep_*=\begin{cases}
          \min\{\delta,\mu_1,1\}& \text{if } N_*=1,\\
          \min\{\delta, \mu_{N_*}-\mu_{N_*-1}, \mu_{N_*+1}-\mu_{N_*}\}& \text{if } N_*\ge 2. 
         \end{cases}
\eeq
\end{theorem}

\bigskip
We make a couple of notes on the formulas of $\xi_n$'s.

(a) In case $n=1$, we set $\chi_1=0$, and use the convention that the last term on the right-hand side of \eqref{xin1} vanishes, then formula \eqref{xin1} agrees with \eqref{xi1}, and hence holds also for $n=1$.

(b) The relation between $\phi_n$'s and $\xi_n$'s is one-to-one. Indeed, the $\phi_n$'s can be solved recursively from \eqref{xi1} and \eqref{xin1} by 
\beq\label{phisolve}
\left\{
\begin{aligned}
& \phi_1=A \xi_1, \\ 
 &\phi_n=A \xi_n -\chi_n + \sum_{\stackrel{k,m\ge 1,}{\mu_k+\mu_m=\mu_n}}B(\xi_k,\xi_m), \quad n \ge 2.
 \end{aligned}
 \right.
 \eeq
where the $\chi_n$'s are still defined by \eqref{chin}.

\medskip
The fact that we can have $\alpha$ fixed instead of requiring \eqref{ffinite} to hold for all $\alpha>0$ comes from the following regularity property of $\xi_n$'s.

\begin{lemma}\label{xireg}
 Let $\phi_n$ and $\xi_n$, for $1\le n\le N_*$, be as in Theorem \ref{finitetheo}. Then
 \beq\label{ximore}
 \xi_n\in  G_{\alpha+1,\sigma}\quad \forall n=1,2,\ldots N_*.
 \eeq
\end{lemma}
\begin{proof}
We prove \eqref{ximore} by induction. 

When $n=1$, since $\phi_1\in G_{\alpha,\sigma}$, we have $\xi_1\in A^{-1}\phi_1\in G_{\alpha+1,\sigma}$.

Let $1\le n<N_*$, and assume that  all $\xi_1,\ldots,\xi_n\in G_{\alpha+1,\sigma}$. It implies that 
$$\chi_{n+1}\in G_{\alpha+1,\sigma}\subset G_{\alpha,\sigma}.$$
This and the fact $\phi_{n+1}\in G_{\alpha,\sigma}$  yield $A^{-1}(\phi_{n+1}+\chi_{n+1})\in G_{\alpha+1,\sigma}$.

For $1\le k,m\le n$, we have from the induction hypothesis that $\xi_k,\xi_m\in G_{\alpha+1,\sigma}$. Then, by \eqref{AalphaB},
 \beqs
 B(\xi_k,\xi_m)\in G_{\alpha+1-1/2,\sigma}=G_{\alpha+1/2,\sigma},
 \eeqs
which yields
\beqs
 A^{-1}B(\xi_k,\xi_j)\in G_{\alpha+3/2,\sigma}\subset G_{\alpha+1,\sigma}.
 \eeqs
 
 Therefore,
\beqs  
 \xi_{n+1}=A^{-1}(\phi_{n+1} + \chi_{n+1}) - \sum_{\stackrel{1\le k,m \le n,}{\mu_k+\mu_m=\mu_{n+1}}}A^{-1}B(\xi_k,\xi_m) \in G_{\alpha+1,\sigma}.
\eeqs 

By the induction principle, $\xi_n\in G_{\alpha+1,\sigma}$ for all $1\le n\le N_*$.
\end{proof}

Before proceeding with the proof of Theorem \ref{finitetheo}, we observe from \eqref{ffinite} that if $1\le N<N_*$ then
\begin{align*}
\Big|f(t)-\sum_{n=1}^N \phi_n t^{-\mu_n}\Big|_{\alpha,\sigma}
&\le \Big|f(t)-\sum_{n=1}^{N_*} \phi_n t^{-\mu_n}\Big|_{\alpha,\sigma}+ \Big|\sum_{n=N+1}^{N_*} \phi_n t^{-\mu_n}\Big| \\
&=\mathcal O(t^{-\mu_{N_*}-\delta})+\mathcal O(t^{-\mu_{N+1}}) \text{ as }t\to\infty,
\end{align*}
hence,
\beq\label{frate}
\Big|f(t)-\sum_{n=1}^N \phi_n t^{-\mu_n}\Big|_{\alpha,\sigma}
=\mathcal O(t^{-\mu_{N+1}}).
\eeq

Thus,  one has, for $1\le N\le N_*$, that
\beq\label{Fcond}
\Big|f(t)-\sum_{n=1}^N \phi_n t^{-\mu_n}\Big|_{\alpha,\sigma}
=\mathcal O(t^{-\mu_N-\delta_N}), 
\eeq
where 
\beq\label{delN}
\delta_N=\begin{cases}
          \mu_{N+1}-\mu_N&\text{for }1\le N<N_*,\\
          \delta&\text{for } N=N_*.
         \end{cases}
\eeq
 
\begin{proof}[Proof of Theorem \ref{finitetheo}]
(i) We first prove that if $N$ is any integer in $[1,N_*]$, then there exists a number $\varep_N>0$ such that for all $\rho \in (0,1)$
 \beq\label{remdelta}
\Big |u(t)-\sum_{n=1}^N \xi_n t^{-\mu_n}\Big |_{\alpha + 1 -\rho,\sigma} =\mathcal O(t^{-\mu_N-\varep_N} )\quad\text{as } t\to\infty.
 \eeq

\textit{Proof of \eqref{remdelta}.} 
We use the following notation. For an integer $n\in[1,N_*]$, define 
\begin{align*}
&F_n(t)=\phi_n t^{-\mu_n},\quad \bar F_n(t)=\sum_{j=1}^n F_j(t),\quad \text{and}\quad \tilde F_n(t)=f(t)-\bar F_n(t),\\
&u_n(t)=\xi_n t^{-\mu_n},\quad \bar u_n(t)=\sum_{j=1}^n u_j(t),\quad\text{and}\quad v_n=u(t)-\bar u_n(t).
\end{align*}

In calculations below, all differential equations hold in $V'$-valued distribution sense on $(T,\infty)$ for any $T>0$, which is similar to \eqref{varform}.
One can easily verify them by using \eqref{Bweak}, and the facts $u\in L^2_{\rm loc}([0,\infty),V)$ and $u'\in L^1_{\rm loc}([0,\infty),V')$ in Definition \ref{lhdef}.

We prove \eqref{remdelta} by induction in $N$.

\medskip
\textit{First step: $N=1$.} Let $w_1(t)=t^{\mu_1} u(t)$. 
We have, for $t>0$, that
\beq\label{wj}
w_1'(t) + A w_1(t) = \phi_1 + H_1(t),
\eeq
where
\beqs
H_1(t)=t^{\mu_1} [ \tilde F_1(t) - B(u(t),u(t))]+ \mu_1 t^{\mu_1-1}u(t).
\eeqs

(a) We estimate $|H_1(t)|_{\alpha,\sigma}$. Equation \eqref{Fcond}, for $N=1$, particularly reads
\beq\label{fphi1}
|f(t)- \phi_1 t^{-\mu_1}|_{\alpha,\sigma}=\mathcal O(t^{-\mu_1-\delta_1} ).
\eeq
It follows  that
\beq\label{fsure}
|f(t)|_{\alpha,\sigma}=\mathcal O(t^{-\mu_1}).
\eeq

Thanks to \eqref{fsure}, we can apply Theorem \ref{theo23} with $\varep=1/2$ and obtain from  inequalities \eqref{uep} and \eqref{Buep} that, as $t\to\infty$,
\begin{align}\label{newu}
 |u(t)|_{\alpha+1/2,\sigma}&=\mathcal O(t^{-\mu_1}),\\
 \label{newB}
 |B(u(t),u(t))|_{\alpha,\sigma}&=\mathcal O(t^{-2\mu_1}).
\end{align}

Combining estimates \eqref{fphi1},  \eqref{newu} and \eqref{newB}, we deduce that there exist $T_0>0$ and $D_0>0$ such that, for $t\ge 0$, 
\begin{align*}
(T_0+t)^{\mu_1} |\tilde F_1(T_0+t)|_{\alpha,\sigma}
&\le D_0(1+t)^{-\delta_1},\\
 (T_0+t)^{\mu_1-1}|u(T_0+t)|_{\alpha+1/2,\sigma} 
 &\le D_0 (1+t)^{-1 },\\
(T_0+t)^{\mu_1} |B(u(T_0+t),u(T_0+t))|_{\alpha,\sigma} 
&\le D_0 (1+t)^{-\mu_1} .
 \end{align*}
Thus, we have
\beqs
|H_1(T_0+t)|_{\alpha,\sigma} \le 3 D_0 (1+t)^{-\varep_1}\quad\forall t\ge0,
\eeqs
where 
\beq \label{dstar}
\varep_1=\min \{ \delta_1,\mu_1,1\}.
\eeq 

(b) Applying Lemma \ref{odelem} to equation \eqref{wj} in $G_{\alpha,\sigma}$ with solution $w_1(T_0+t)$, for $t\in [0,\infty)$, yields 
\beqs
|w_1(T_0+t)-A^{-1}\phi_1 |_{\alpha+1-\rho,\sigma}= \mathcal O (t^{-\varep_1})
\eeqs
for any $\rho \in(0,1)$, and consequently, 
\beqs
|w_1(t)-A^{-1}\phi_1 |_{\alpha+1-\rho,\sigma}= \mathcal O (t^{-\varep_1}).
\eeqs
Multiplying this equation by $t^{-\mu_1}$ yields
\beqs\label{rem1}
|u(t)-\xi_1  t^{-\mu_1}|_{\alpha+1-\rho,\sigma}= \mathcal O (t^{-\mu_1-\varep_1})\quad\forall\rho\in(0,1).
\eeqs
This proves that
\beq\label{st1}
\text{the statement \eqref{remdelta} holds true for $N=1$ with $\varep_1$ defined by \eqref{dstar}.} 
\eeq 

\medskip
\textit{Induction step.} Let $N$ be an integer with $1\le N<N_*$, and assume  there exists $\varep_N>0$ such that 
\beq\label{vNrate}
|v_N(t)|_{\alpha+1-\rho,\sigma}=\mathcal O(t^{-\mu_N-\varep_N})\quad \forall \rho \in (0,1). 
\eeq

\medskip
(a) Equation for $v_N$. Since $v_N'=u'-\bar u_N'$, we calculate $u'$ and $\bar u_N'$ in terms of the quantities that are more appropriate for the analysis of $v_N$.

$\bullet$  Calculating $u'$. By NSE,
\begin{align*}
u'&=-Au-B(u,u)+f(t)\\
&=-Av_N - A\bar u_N -B(\bar u_N+v_N,\bar u_N+v_N)+\bar F_N+F_{N+1}  +\tilde F_{N+1},\end{align*}
hence,
\beq \label{upr}
u'=-Av_N- A\bar u_N+\bar F_N-B(\bar u_N,\bar u_N)+ \phi_{N+1}t^{-\mu_{N+1}} +h_{N+1,1},
\eeq 
 where
\beqs
h_{N+1,1}=-B(\bar u_N,v_N)-B(v_N,\bar u_N)-B(v_N,v_N)+\tilde F_{N+1}.
\eeqs

Firstly, note that
\beqs
-A\bar u_N+\bar F_N= - \sum_{n=1}^N \frac1{t^{\mu_n}}\Big( A\xi_n-\phi_n\Big).
\eeqs

Secondly, we write
\begin{align*}
B(\bar u_N,\bar u_N)
&=\sum_{ m,j=1}^N t^{-\mu_m-\mu_j}B(\xi_m,\xi_j)\\
&=\sum_{n=1}^{N}\frac1{t^{\mu_n}}\Big( \sum_{\stackrel{1\le m,j\le N,}{\mu_m+\mu_j=\mu_n}}B(\xi_m,\xi_j)\Big)
+\frac1{t^{\mu_{N+1}}} \sum_{\stackrel{1\le m,j\le N,}{\mu_m+\mu_j=\mu_{N+1}}}B(\xi_m,\xi_j)
+h_{N+1,2},
\end{align*}
where
\beqs
h_{N+1,2}=\sum_{\stackrel{1\le m,j\le N,}{\mu_m+\mu_j\ge \mu_{N+2}}} t^{-\mu_m-\mu_j}B(\xi_m,\xi_j).
\eeqs

$\bullet$ Calculating $\bar u_N'$. Note that $\mu_N+1 \ge \mu_{N+1}$ and
$$ \{ \mu_p + 1 : 1 \le p \le N  \} \cap [\mu_1, \mu_{N+1}] \subset \{\mu_k: 1 \le k \le N+1\}.$$
Thus,
\beq\label{uNp}
-\bar u_N'=\sum_{p=1}^N \frac{\mu_p\xi_p}{t^{\mu_{p}+1}}=\sum_{n=1}^{N} \frac{\chi_n}{t^{\mu_n}}+\frac{\chi_{N+1}}{t^{\mu_{N+1}}}
+h_{N+1,3},
\eeq
where
\beqs
h_{N+1,3}=\sum_{\stackrel{1\le p\le N,}{\mu_{p}+1\ge \mu_{N+2}}} \frac{\mu_p\xi_p}{t^{\mu_{p}+1}}.
\eeqs

$\bullet$  Combining the above equations \eqref{upr}--\eqref{uNp} yields 
\begin{align*}
v_N'&=-Av_N +\frac1{t^{\mu_{N+1}}}\Big(- \sum_{\stackrel{1\le m,j\le N,}{\mu_m+\mu_j=\mu_{N+1}}}B(\xi_m,\xi_j)+\phi_{N+1}+\chi_{N+1}\Big)\\
&\quad - \sum_{n=1}^N \frac1{t^{\mu_n}}\Big( A\xi_n+\sum_{\stackrel{1\le m,j\le N,}{\mu_m+\mu_j=\mu_n}}B(\xi_m,\xi_j)-\phi_n  - \chi_n\Big)+h_{N+1,4},
\end{align*}
where
\beq\label{h4}
h_{N+1,4}=h_{N+1,1}-h_{N+1,2}+h_{N+1,3}.
\eeq

 Note, for $1\le n\le N+1$, that
 \beqs
 \sum_{\stackrel{1\le m,j\le N,}{\mu_m+\mu_j=\mu_n}}B(\xi_m,\xi_j)=\sum_{\stackrel{1\le  m,j \le n-1,}{\mu_m+\mu_j=\mu_n}}B(\xi_m,\xi_j).
 \eeqs 

Therefore,   one has, for $1\le n\le N$,
\beqs
 A\xi_n+\sum_{\stackrel{1\le m,j \le N,}{\mu_m+\mu_j=\mu_n}}B(\xi_m,\xi_j)-\phi_n-\chi_n =0,
\eeqs
and 
\beqs 
-\sum_{\stackrel{1\le m,j \le N,}{\mu_m+\mu_j=\mu_{N+1}}}B(\xi_m,\xi_j)+\phi_{N+1}+\chi_{N+1}=A\xi_{N+1}.
\eeqs

These yield
\beq\label{vNp}
v_N'=-Av_N + t^{-\mu_{N+1}}A\xi_{N+1}+h_{N+1,4}(t).
\eeq

(b) Define $w_{N+1}(t)=t^{\mu_{N+1}} v_N(t)$  for $t>0$. 
We have
\beqs
 w_{N+1}'=t^{\mu_{N+1}}v_N' + \mu_{N+1}t^{\mu_{N+1}-1} v_N,
\eeqs 
which, thanks to \eqref{vNp}, yields
\beq \label{wNeq}
w_{N+1}'=-Aw_{N+1}+A\xi_{N+1}+H_{N+1}(t),
\eeq
where
\beq\label{HN1}
H_{N+1}(t)
= t^{\mu_{N+1}}h_{N+1,4}(t)+ \mu_{N+1}t^{\mu_{N+1}-1} v_N(t).
\eeq

Let $\rho$ be any number in the interval $(0,1)$.

\medskip
(c) We estimate $H_{N+1}(t)$ now.
From \eqref{vNrate} and the fact $\mu_N+1\ge \mu_{N+1}$, it follows 
\beq\label{vNrate2}
t^{\mu_{N+1}-1}  | v_N(t)|_{\alpha+1-\rho,\sigma}=\mathcal O(t^{\mu_{N+1}-1}t^{-\mu_N-\varep_N})=\mathcal O(t^{-\varep_N}).
\eeq

By \eqref{Fcond}, we have
\beqs \label{Ftil}
t^{\mu_{N+1}}|\tilde F_{N+1}(t)|_{\alpha,\sigma}=\mathcal O(t^{-\delta_{N+1}}).
\eeqs

 Clearly,
\beq \label{ubu}
|\bar u_N(t)|_{\alpha+1,\sigma}=\mathcal O(t^{-\mu_1}).
\eeq

By inequality \eqref{AalphaB}, estimate \eqref{vNrate} for $\rho=1/2$, and \eqref{ubu}, it follows that
\begin{align*}
&t^{\mu_{N+1}}|B(v_N(t),\bar u_N(t))|_{\alpha,\sigma},\ 
t^{\mu_{N+1}}|B(\bar u_N(t),v_N(t))|_{\alpha,\sigma}
 =\mathcal O(t^{\mu_{N+1}}t^{-\mu_{N}-\varep_N} t^{-\mu_1})=\mathcal O(t^{-\varep_N}),\\
&t^{\mu_{N+1}}|B(v_N(t), v_N(t))|_{\alpha,\sigma}=\mathcal O(t^{\mu_{N+1}}t^{-\mu_N-\varep_N}t^{-\mu_N-\varep_N})=\mathcal O(t^{-\varep_N}).
\end{align*}
Above, we used the fact $2\mu_N\ge \mu_N+\mu_1\ge \mu_{N+1}$.
Hence,
\beqs 
 t^{\mu_{N+1}} |h_{N+1,1}(t)|_{\alpha,\sigma}=\mathcal O(t^{-\delta_{N+1}}) +\mathcal O(t^{-\varep_N}).
\eeqs

It is obvious that
\beqs
 t^{\mu_{N+1}}\sum_{\stackrel{1\le m,j\le N,}{\mu_m+\mu_j\ge  \mu_{N+2}}} t^{-\mu_m}t^{-\mu_j}|B(\xi_m,\xi_j)|_{\alpha,\sigma} 
=\mathcal O(t^{-(\mu_{N+2}-\mu_{N+1})}),
\eeqs
and thus, 
\beqs 
 t^{\mu_{N+1}} |h_{N+1,2}(t)|_{\alpha,\sigma}=\mathcal O(t^{-(\mu_{N+2}-\mu_{N+1})}).
\eeqs

It is also clear that 
\beqs 
 t^{\mu_{N+1}} |h_{N+1,3}(t)|_{\alpha,\sigma}=\mathcal O(t^{-(\mu_{N+2}-\mu_{N+1})}).
\eeqs

Combining these estimates of $t^{\mu_{N+1}}h_{N+1,j}(t)$ for $j=1,2,3$, with \eqref{HN1}, \eqref{h4} and \eqref{vNrate2} gives 
\beq \label{HNt-est}
|H_{N+1}(t)|_{\alpha,\sigma}=\mathcal O(t^{-\delta_{N+1}})+\mathcal O(t^{-\varep_N})+\mathcal O(t^{-(\mu_{N+2}-\mu_{N+1})})
=\mathcal O(t^{-\varep_{N+1}}),
\eeq
where
\beq\label{bN}
\varep_{N+1}=\min\{ \delta_{N+1},\varep_N,\mu_{N+2}-\mu_{N+1}\}.
\eeq

(d) Note that from Lemma \ref{xireg} that $A\xi_{N+1}\in G_{\alpha,\sigma}$. 
By applying Lemma \ref{odelem} to equation \eqref{wNeq} and solution $w_{N+1}(T_1+t)$ for some sufficiently large $T_1>0$ with the use of \eqref{HNt-est}, we obtain 
\beqs
|w_{N+1}(T_1+t)-A^{-1}(A\xi_{N+1})|_{\alpha+1-\rho,\sigma}\le  C(1+t)^{-\varep_{N+1}}\quad\forall t\ge 1.
\eeqs
Thus,
\beqs
|w_{N+1}(t)-\xi_{N+1}|_{\alpha+1-\rho,\sigma}=  \mathcal O(t^{-\varep_{N+1}}).
\eeqs

Multiplying this equation by $t^{-\mu_{N+1}}$ yields
\beqs
|v_N(t)- \xi_{N+1}t^{-\mu_{N+1}}|_{\alpha+1-\rho,\sigma}=\mathcal O(t^{-\mu_{N+1}-\varep_{N+1}}),
\eeqs
that is,
\beqs
|v_{N+1}(t)|_{\alpha+1-\rho,\sigma}=\mathcal O(t^{-\mu_{N+1}-\varep_{N+1}}).
\eeqs
This proves 
\beq\label{stN} 
\text{the statement \eqref{remdelta} holds for $N:=N+1$ with  $\varep_{N+1}$ defined by \eqref{bN}.}
\eeq 

By the induction principle, we have \eqref{remdelta} holds true for all $N=1,2,\ldots,N_*$. This completes the proof of \eqref{remdelta}.

\medskip
(ii) We now prove \eqref{ufinite}.

\smallskip
\textit{Case $N_*=1$.} We have in this case, thanks to \eqref{delN}, $\delta_1=\delta$ and $\varep_1$ in \eqref{dstar} equals $\varep_*$ in \eqref{epstar}. Thus, the statement \eqref{ufinite} just follows \eqref{st1}.

\smallskip
\textit{Case $N_*\ge 2$.} Similar to \eqref{frate}, one has from \eqref{remdelta}, for $N=N_*$, that 
\beq\label{vminus}
|v_{N_*-1}(t)|_{\alpha+1-\rho,\sigma}=\mathcal O(t^{-\mu_{N_*}}).
\eeq

We repeat the induction step in part (i) for $N=N_*-1$, but now with 
$$\delta_{N+1}=\delta_{N_*}=\delta\text{ and, thanks to \eqref{vminus}, }\varep_N=\varep_{N_*-1}=\mu_{N_*}-\mu_{N_*-1}.$$

Then one obtains from \eqref{stN} that
\beqs
|v_{N_*}(t)|_{\alpha+1-\rho,\sigma}=\mathcal O(t^{-\mu_{N_*}-\varep_{N_*}}), 
\eeqs
where, according to formula \eqref{bN}, $\varep_{N_*}=\varep_{N+1}=\min\{ \delta,\mu_{N_*}-\mu_{N_*-1},\mu_{N_*+1}-\mu_{N_*}\}$ which exactly is  $\varep_*$.
This completes our proof.
\end{proof}

\subsection{Infinite expansions}\label{infinitesub}

We focus, in this subsection, the case when the force has an infinite expansion, and obtain an infinite expansion for any Leray-Hopf weak solution of the NSE. The force's expansion considered will be of the following type.

\textbf{(B2)} \textit{Suppose there exist  real numbers $\sigma\ge 0$, $\alpha\ge 1/2$,  a strictly increasing, divergent sequence of positive numbers $(\gamma_n)_{n=1}^\infty$ and a sequence of functions  $(\psi_n)_{n=1}^\infty$ in $G_{\alpha,\sigma}$ such that, in the sense of Definition \ref{expanddef},
\beq\label{fseq}
f(t)\polsim \sum_{n=1}^\infty \psi_n t^{-\gamma_n} \text{ in }G_{\alpha,\sigma}.
\eeq
}

Similar to the previous subsection, the appropriate set of powers generated  by $\gamma_n$'s and $1$ is  
\begin{align*}
S_\infty=\Big \{  (\sum_{j=1}^p \gamma_{n_j}) + k: \text{ for some } p\ge 1,\ n_1,n_2,\ldots,n_p\ge 1, 
\text{ and some integer } k\ge 0\Big  \}. 
\end{align*}

Then $S_\infty\subset (0,\infty)$, and property \eqref{Sprop} still holds with $S_\infty$ replacing $S_*$.
Since $\gamma_n\to\infty$ as $n\to\infty$, we can order $S_\infty$, and denote
\beq 
S_\infty=\{ \mu_n:  n\in \N\} \text{ with $\mu_n$'s being strictly increasing.}
\eeq
(This is possible by ordering finitely many elements in each set $S_\infty\cap (n-1,n]$, for all $n\in\N$.)
Note that $\mu_n\to\infty$ as $n\to\infty$.

Then rewrite \eqref{fseq} as
\beq \label{fexp2}
f(t)\polsim \sum_{n=1}^{\infty} \phi_n t^{-\mu_n }\quad \text{in }G_{\alpha,\sigma}\quad\text{as } t\to\infty,
\eeq
where the sequence $(\phi_n)_{n=1}^\infty$ in $G_{\alpha,\sigma}$ is defined by $\phi_n=\psi_k$ if there exists $k\ge 1$ such that $\mu_n=\gamma_k$, and $\phi_n=0$ otherwise.

By the same arguments as in subsection \ref{finsubsec}, the  estimate \eqref{frate} now holds for all $N\geq1$.


\begin{theorem} \label{mainthm}
Assume {\rm (B2)} and the corresponding expansion \eqref{fexp2}.
Then any Leray-Hopf weak solution $u(t)$  of \eqref{fctnse} and \eqref{uzero}
has the asymptotic expansion
 \beq\label{uexpand}
u(t)\polsim  \sum_{n=1}^\infty  \xi_n t^{-\mu_n}\quad \text{in }G_{\alpha+1-\rho,\sigma},\quad \forall \rho \in (0,1),
 \eeq
 where $\xi_n$ is defined by \eqref{xi1} for $n=1$, and by \eqref{xin1} for $n\ge 2$.
  More precisely, one has for any $N\ge 1$ that
 \beq\label{urem}
\Big |u(t)- \sum_{n=1}^N  \xi_n t^{-\mu_n}\Big|_{\alpha+1-\rho,\sigma}=\mathcal O(t^{-\mu_{N+1}})\quad \text{ as }t\to\infty,
\quad \forall \rho \in (0,1).
 \eeq
\end{theorem}
\begin{proof}
Clearly, $f$ satisfies condition (B1)  for any $N_0\ge 1$, and \eqref{ffinite}
 for any $N_*\ge 1$.
Hence, applying Theorem \ref{finitetheo} for each $N_*\ge 1$, we obtain the expansion \eqref{uexpand} for $u(t)$. 
Then similar to \eqref{frate}, we obtain from \eqref{uexpand} that \eqref{urem} holds for all $N\ge 1$.
\end{proof}

\section{Properties of the expansions}
\label{furthersec}

According to Theorem \ref{mainthm}, for each force $f$ satisfying the required conditions, there exists a unique sequence $(\xi_n)_{n=1}^\infty$ such that the expansion \eqref{uexpand} holds for all Leray-Hopf weak solution $u(t)$. The first part of this section investigates the range of $(\xi_n)_{n=1}^\infty$ when the force $f$ varies.

\medskip
Below, we focus on the infinite expansions in section \ref{infinitesub}, with  $\gamma_n=n$ for all $n\in\N$ in Assumption (B2), which implies that $\mu_n=n$ for all $n\in\N$. 
In this case, \eqref{fexp2} and \eqref{uexpand} read as
\beq\label{fncase}
f(t)\polsim \sum_{n=1}^\infty \phi_n t^{-n} \quad \text{in }G_{\alpha,\sigma},
\eeq
\beq\label{uncase}
u(t)\polsim \sum_{n=1}^\infty \xi_n t^{-n} \quad \text{in }G_{\alpha+1-\rho,\sigma}, \quad \forall\rho\in(0,1),
\eeq
where $\phi_n$'s and $\xi_n$'s, referring to \eqref{xi1} and \eqref{xin1}, are related by 
\beq\label{xincase}
\left\{
\begin{aligned}
& \xi_1=A^{-1} \phi_1, \\ 
 &\xi_n=A^{-1}\Big[ \phi_n +(n-1)\xi_{n-1} - \sum_{k=1}^{n-1}B(\xi_k,\xi_{n-k})\Big], \quad n \ge 2.
 \end{aligned}
 \right.
 \eeq
Above, we used the fact that $\chi_n$ given by \eqref{chin} now is $(n-1)\xi_{n-1}$ for all $n\ge 2$.

We recall note (b) after Theorem \ref{finitetheo} that the relation between $(\phi_n)_{n=1}^\infty$ and $(\xi_n)_{n=1}^\infty$ in \eqref{xincase} is one-to-one, and \eqref{phisolve} now reads as  
\beq\label{phi-xi}
\left\{
\begin{aligned}
& \phi_1=A \xi_1, \\ 
 &\phi_n=A \xi_n -(n-1)\xi_{n-1} + \sum_{k=1}^{n-1}B(\xi_k,\xi_{n-k}), \quad n \ge 2.
 \end{aligned}
 \right.
 \eeq

The following proposition gives a sufficient condition for $(\xi_n)_{n=1}^\infty$ so that $\sum_{n=1}^{\infty} \xi_n t^{-n}$ is an expansion of a Leray-Hopf weak solution with some force $f$ as in \eqref{fncase}.

\begin{theorem}\label{conv-series}
Let $(c_n)_{n=1}^\infty$ be a sequence of non-negative numbers such that 
\beq\label{dn}
\sum_{n=2}^\infty n d_n<\infty,\text{ where } d_n= \max\{ c_k c_{n-k}: 1\le k\le n-1\}.
\eeq

Given $\alpha\ge 1/2$ and $\sigma\ge 0$.
Suppose  $(\zeta_n)_{n=1}^{\infty}$ is a sequence in $G_{\alpha+1,\sigma}$ such that
\beq \label{xin-dn}
|\zeta_n|_{\alpha+1,\sigma}\le c_n \quad\forall n\in\N. 
\eeq

Then there exists a forcing function $f(t)$ such that any Leray-Hopf weak solution $u(t)$ of \eqref{fctnse} and \eqref{uzero} satisfies 
\beq\label{xran2}
 u(t) \polsim \sum_{n=1}^{\infty} \zeta_n t^{-n} \quad \text{ in } G_{\alpha+1-\rho,\sigma},\quad\forall \rho\in(0,1).
\eeq
Moreover, the series of the expansion, $\sum_{n=1}^{\infty} \zeta_n t^{-n}$, converges in $G_{\alpha+1,\sigma}$ absolutely and uniformly in $[1,\infty)$.
\end{theorem}
\begin{proof}
In case $c_n=0$ for all $n$, then $\zeta_n=0$ for all $n$. We simply take $f=0$, which gives $\phi_n=0$ for all $n$.
Then we have expansion \eqref{uncase}, where the $\xi_n$'s are given by \eqref{xincase}, which obviously yields $\xi_n=0=\zeta_n$ for all $n$. Hence \eqref{xran2} follows \eqref{uncase}.

We now consider the case that there exists $n_0\in\N$ such that $c_{n_0}>0$. 
Clearly, from the definition of $d_n$ one has   $c_{n_0}c_{n} \le d_{n+n_0}$ for $n \ge 1$. Using this and \eqref{dn} yield
\beq \label{series-cn}
 \sum_{n=1}^\infty  c_n \le  \sum_{n=1}^\infty n c_n \le \sum_{n=1}^\infty n \frac{d_{n+n_0}}{c_{n_0}}  \le \frac{1}{c_{n_0}} \sum_{n=1}^\infty (n+n_0) d_{n+n_0} < \infty.
\eeq

Define
\beq\label{phi-zeta}
\left\{
\begin{aligned}
& \phi_1=A \zeta_1, \\ 
&\phi_n=A \zeta_n -(n-1)\zeta_{n-1} + \sum_{k=1}^{n-1}B(\zeta_k,\zeta_{n-k}), \quad n \ge 2.
 \end{aligned}
\right.
 \eeq
(This, in fact, is the construction of $\phi_n$'s in \eqref{phi-xi} with $\xi_n$'s being replaced with $\zeta_n$'s.)

We estimate, for $n=1$,
$$|\phi_1|_{\alpha,\sigma}=|A\zeta_1|_{\alpha,\sigma}=|\zeta_1|_{\alpha+1,\sigma}\le c_1.$$

For $n\ge 2$, we have from \eqref{phi-zeta} and \eqref{AalphaB} that
\begin{align*}
 |\phi_n|_{\alpha,\sigma} 
 &\le |A \zeta_n|_{\alpha,\sigma} +(n-1)|\zeta_{n-1}|_{\alpha,\sigma} + \sum_{k=1}^{n-1}|B(\zeta_k,\zeta_{n-k})|_{\alpha,\sigma}\\
 &\le | \zeta_n|_{\alpha+1,\sigma} +(n-1)|\zeta_{n-1}|_{\alpha+1,\sigma} + K^\alpha \sum_{k=1}^{n-1}|\zeta_k|_{\alpha+1/2,\sigma}|\zeta_{n-k}|_{\alpha+1/2,\sigma}.
\end{align*}
Then, by \eqref{xin-dn},
\begin{align*}
 |\phi_n|_{\alpha,\sigma} 
  &\le  c_n +  (n-1)c_{n-1}+ K^\alpha  \sum_{k=1}^{n-1}c_kc_{n-k} \\
 &\le    c_n +  (n-1)c_{n-1} + K^\alpha (n-1)d_n < \infty.
\end{align*}
Therefore, by \eqref{series-cn} and \eqref{dn},
\beqs 
\sum_{n=1}^\infty |\phi_n|_{\alpha,\sigma} \le \sum_{n=1}^\infty c_n + \sum_{n=1
}^\infty n c_n +  K^{\alpha}\sum_{n=2}^\infty(n-1)d_n < \infty.
\eeqs

It follows that $\sum_{n=1}^{\infty} \phi_n t^{-n}$ converges in $G_{\alpha,\sigma}$ absolutely and uniformly on $[1,\infty)$. Thus, we can define $f(t)$ for $t\ge 0$ as following:
\beq
f(t)=\begin{cases}
\sum_{n=1}^{\infty} \phi_n&\text{ if }   0\le t<1,\\   
\sum_{n=1}^{\infty} \phi_n t^{-n}&\text{ if }   t\ge 1.    
     \end{cases}
\eeq 

Clearly, $f$ satisfies (A), and $f(t)\polsim \sum_{n=1}^{\infty} \phi_n t^{-n}$ in $G_{\alpha,\sigma}$, hence $f$ satisfies (B2) too.

Let $u$ be a Leray-Hopf weak solution of \eqref{fctnse} and \eqref{uzero}. Applying Theorem \ref{mainthm} gives the expansion \eqref{uncase} for $u(t)$, where the $\xi_n$'s are given by \eqref{xincase}. 

Solving for $\zeta_n$'s from \eqref{phi-zeta} gives
\beq\label{zetan}
\left\{
\begin{aligned}
& \zeta_1=A^{-1} \phi_1, \\ 
 &\zeta_n=A^{-1}\Big[ \phi_n +(n-1)\zeta_{n-1} - \sum_{k=1}^{n-1}B(\zeta_k,\zeta_{n-k})\Big], \quad n \ge 2.
 \end{aligned}
 \right.
 \eeq

Comparing \eqref{xincase} and \eqref{zetan} shows $\xi_n=\zeta_n$ for all $n\in\N$. Therefore, \eqref{xran2} follows \eqref{uncase}. 

By \eqref{xin-dn}, we have, for all $t  \ge 1$ and $n \ge 1$, that
\beq 
|\zeta_n t^{-n}|_{\alpha+1,\sigma} \le |\zeta_n |_{\alpha+1,\sigma} \le c_n.
\eeq
This estimate and  \eqref{series-cn} imply that  $\sum_{n=1}^{\infty} \zeta_n t^{-n}$ converges in $G_{\alpha+1,\sigma}$ absolutely and uniformly  on $[1,\infty)$. 
The proof is complete.
\end{proof}

\begin{example}\label{cn}
In Theorem \ref{conv-series}, assume there exist $M>0$, $\lambda>2$ and  $N_0\ge 1$ such that $c_n\le M n^{-\lambda}$ for all $n\ge N_0$. Then the sequence $(c_n)_{n=1}^\infty$ is bounded by, say, a number $c_*>0$. Let $n\ge 2N_0$.
If $1\le k\le n/2$, then $n-k\ge n/2\ge N_0$ and 
$$c_k c_{n-k}\le c_*c_{n-k}\le c_* M(n-k)^{-\lambda}\le c_* M(n/2)^{-\lambda}.$$
If $n/2<k<n$, then $k> N_0$ and 
$$c_k c_{n-k}\le c_k c_*\le c_* M k^{-\lambda}\le c_* M(n/2)^{-\lambda}.$$
Hence, $d_n\le  2^{\lambda} M n^{-\lambda}$ for all $n\ge 2N_0$. 
Therefore, condition \eqref{dn} is satisfied.

\end{example}

As a special case of Theorem \ref{conv-series}, the next corollary shows that the expansion of $u(t)$, essentially, can  be any finite sum in $G_{\alpha+1,\sigma}$ (of course, of the same type.)

\begin{corollary}\label{finite}
Let  $\alpha\ge 1/2$, $\sigma\ge 0$ be given numbers, and $\zeta_1,\ldots,\zeta_N$ be given elements in $G_{\alpha+1,\sigma}$, for some $N\ge 1$. 
Then there exists a forcing function $f(t)$ such that any Leray-Hopf weak solution $u(t)$ of \eqref{fctnse} and \eqref{uzero} has the expansion $u(t)\polsim \sum_{n=1}^{\infty} \xi_n t^{-n}$ in $G_{\alpha+1-\rho,\sigma}$ for all $\rho\in(0,1)$, where $\xi_n=\zeta_n$ for $1\le n\le N$ and $\xi_n=0$ for all $n>N$. 
 \end{corollary}
\begin{proof}
For $n>N$, set $\zeta_n=0$. Define $c_n=|\zeta_n|_{\alpha+1,\sigma}$ for $1\le n\le N$, and $c_n=0$ for $n>N$.
Let $d_n$ be defined as in \eqref{dn}. One can verify that $d_n=0$ for all $n>2N$.
Hence, the condition \eqref{dn} is satisfied. Then the conclusion of this corollary follows Theorem \ref{conv-series}.
(In fact, following the construction of $f(t)$, we have $\phi_n=0$ for $n>2N$, and $f(t)$,  for $t\ge 1$, is simply the finite sum $\sum_{n=1}^{2N}\phi_n t^{-n}$.)
\end{proof}

\begin{example}[Divergent expansions]
 For a given and fixed force $f$, the expansion \eqref{uexpand} may not converge. We give here a simple example. Let $\phi_1\ne 0$ be a function in $R_1H$ such that $B(\phi_1,\phi_1)=0$, and let $\phi_n=0$ for all $n\ge 2$. (For e.g., $\phi_1(\vecx)=\varep \vece_2[e^{i\vece_1\cdot \vecx}+e^{-i\vece_1\cdot \vecx}]$ for any $\varep>0$.)
 From \eqref{xincase}, one can easily verify that $\xi_n=(n-1)!\phi_1$ for all $n\in\N$. Hence, the expansion $\sum_{n=1}^\infty \xi_n t^{-n}$ is divergent in $H$ for all $t>0$.
\end{example}

\begin{remark}\label{normalcompare}
We recall the normalization map for NSE, in case the force is potential, defined by Foias and Saut \cite{FS87,FS91}.
First, rewrite $\sigma(A)=\{\lambda_n:n\in\N\}$, where $\lambda_n$ is strictly increasing.
For any $u^0\in V$ such that the solution $u(t)$ of \eqref{fctnse} and \eqref{uzero} is regular on $[0,\infty)$, there exists
$\xi=(\xi_n)_{n=1}^\infty$, with $\xi_n\in R_{\lambda_n}H$ for all $n\in\N$, such that the expansion \eqref{expand} of $u(t)$ can be reconstructed explicitly based on $\xi$ only, i.e., $q_n(t)=q_n(t,\xi)$ -- a polynomial in both $t$ and $\xi$. The normalization map is defined as $W(u^0)=\xi= (\xi_n)_{n=1}^\infty$.
Thus, regarding the reconstructions of the asymptotic expansions for solutions, the sequence $(\xi_n)_{n=1}^\infty$ in \eqref{uncase} and $W(u^0)$ have the similar roles, because they totally determine the  expansions \eqref{uncase} and \eqref{expand}, respectively. We now briefly compare (a) their ranges, and (b) the convergence of their generated expansions.

Regarding (a), it is known that given \textit{small} elements $\zeta_n\in R_{\lambda_n} H$, for $n=1,2,\ldots,N$, there exists $u^0$ such that $W(u^0)= (\xi_n)_{n=1}^\infty$ with
$\xi_n=\zeta_n$ for $1\le n\le N$.
However, it is not known what  $\xi_n$'s might be for $n>N$. For the expansion \eqref{uncase}, Theorem \ref{conv-series}, Example \ref{cn} and Corollary \ref{finite} give specific and simple characteristics of the possible values of $(\xi_n)_{n=1}^\infty$.

Regarding (b), it is not known what might be a general, non-trivial $W(u^0)$ such that the expansion \eqref{expand}, when generated by $W(u^0)$, is convergent. (See \cite{FHOZ1,FHOZ2} for more information about this topic.) In contrast, the expansion \eqref{uncase} is just a power series having $\xi_n$'s as its coefficients. Hence, a simple condition such as $\limsup_{n\to\infty} |\xi_n|_{\alpha,\sigma}^{1/n}<\infty$ is enough to conclude that the expansion \eqref{uncase} converges in $G_{\alpha,\sigma}$ for sufficiently large $t$.
\end{remark}

\medskip
The second part of this section investigates the possible type of decay for the Leray-Hopf weak solutions after all the power-decaying terms. 
More specifically, in Theorem \ref{conv-series}, $u(t)\polsim \sum_{n=1}^{\infty} \xi_n t^{-n}$ in $G_{\alpha,\sigma}$ with $ \sum_{n=1}^{\infty} \xi_n t^{-n}$  converging to $\bar u(t)$ in $G_{\alpha+1,\sigma}$ for all $t\ge 1$.
Hence, by this expansion and the theory of power series, 
\beq\label{difdec}
|u(t)-\bar u(t)|_{\alpha,\sigma}=\mathcal O(t^{-\mu}) \quad\forall \mu>0.
\eeq

The next theorem states that, the remainder in \eqref{difdec}, decays faster and, in fact, it decays exponentially.

\begin{theorem}\label{nextrate}
Given $\alpha\ge 1/2$ and $\sigma\ge 0$. Suppose that 
\beq \label{fseries}
f(t)=\sum_{n=1}^\infty \phi_n t^{-n} \quad \text{in }G_{\alpha,\sigma},\quad \forall t\ge T_0, \text{ for some }T_0>0, 
\eeq
where $(\phi_n)_{n=1}^\infty$ is a sequence in $G_{\alpha,\sigma}$.

Let $(\xi_n)_{n=1}^\infty$ be defined by  \eqref{xincase}, and assume  
\beq\label{xseries}
\limsup_{n\to\infty}\Big(|\xi_n|_{\alpha+1,\sigma}^{1/n}\Big)<\infty.
\eeq

Then
\begin{enumerate}[label=\rm (\roman*)]

 \item The series  $\sum_{n=1}^\infty \xi_n t^{-n}$  converges in $G_{\alpha+1,\sigma}$  to a function $\bar u(t)$ for sufficiently large $t$. 

\item If $u(t)$ is a Leray-Hopf weak solution of \eqref{fctnse} and \eqref{uzero}, then one has for all $\rho \in (0,1/2]$ that, as $t \to \infty$, 
 \beq\label{udif}
 |u(t)-\bar u(t)|_{\alpha+\frac{1}{2}-\rho,\sigma}=
 \left\{
 \begin{aligned}
&\mathcal O( e^{-t}), &&\text{ if }\xi_1= 0,\\
&\mathcal O( t^\beta e^{-t})   \text{ for some }\beta>0,&&\text{ if } \xi_1\ne 0.  
 \end{aligned}
 \right.
 \eeq
\end{enumerate}
\end{theorem}
\begin{proof}

(i) This part is a straight consequence of power series theory in Banach spaces, see e.g. \cite[Chapter~IX]{Dieudonne1969}.
Indeed, \eqref{xseries} implies that there is $R>0$ such that
$\sum_{n=1}^\infty \xi_n z^n$ converges in $G_{\alpha+1,\sigma}$ absolutely and uniformly for $z\in[-R,R]$.
Hence, denoting $T_1=1/R$, we have
\beq \label{uu}
\bar u(t)=\sum_{n=1}^\infty \xi_n t^{-n} \text{ converges in $G_{\alpha+1,\sigma}$ absolutely and uniformly for all $t\ge T_1$.}
\eeq 

(ii) One can verify from \eqref{fseries} that $f$ satisfies (A) and (B2). Let  $T_2= \max\{T_0,T_1\}$. 

(a) First, we claim that,  for all $ t > T_2$,
\beq \label{u-bar-eq}
\bar u'(t)+A\bar u(t) +B(\bar u(t),\bar u(t))=f(t) \quad \text{in } G_{\alpha,\sigma}.
\eeq

Indeed, if $t> T_2$ then
\begin{align*}
\bar{u}'(t)&= \sum_{n=1}^{\infty} u_n'(t)=\sum_{n=2}^{\infty}\frac{-(n-1)\xi_{n-1}}{t^{n}} \quad \text{ in } G_{\alpha+1,\sigma}, \\
A\bar{u}(t)&=\sum_{n=1}^{\infty}\frac{A\xi_n}{t^{n}}\quad \text{ in }G_{\alpha,\sigma}. 
\end{align*}

By \eqref{AalphaB}, \eqref{uu}, and Cauchy's product, we infer that 
\beq
 B(\bar{u}(t),\bar{u}(t))=\sum_{n=2}^{\infty}\frac{1}{t^n}\Big[\sum_{k=1}^{n-1}B(\xi_k,\xi_{n-k})\Big]\quad \text{ in $G_{\alpha+1/2,\sigma}$ for $t> T_2$.}
\eeq

Thus, we have for $t> T_2$ that the following identities hold in $G_{\alpha,\sigma}$ 
\begin{align*}
\bar u'(t)+A\bar u(t) +B(\bar u(t),\bar u(t))
&=\frac{A\xi_1}t +  \sum_{n=2}^{\infty}\frac1{t^{n}}\Big\{ -(n-1)\xi_{n-1}+A\xi_n+\sum_{k=1}^{n-1}B(\xi_k,\xi_{n-k})\Big\}\\
&=\sum_{n=1}^\infty \phi_n t^{-n}=f(t).
\end{align*}
This proves  \eqref{u-bar-eq}.

\medskip
(b) Let $\rho \in (0,1/2]$. Let $w=u-\bar u$. 
Suppose there exists $T_*\ge T_2$ such that 
 \beq\label{uKcond}
 K^{\alpha+\frac{1}{2}-\rho}(|u(t)|_{\alpha+1-\rho,\sigma}+|\bar u(t)|_{\alpha+1-\rho,\sigma})<1\quad \forall t>T_*.
 \eeq 

 We claim that,  for $t\ge T_*$,
 \beq\label{w1}
  |w(t)|_{\alpha+1/2-\rho,\sigma}
  \le |w(T_*)|_{\alpha+1/2-\rho,\sigma} e^{-t+T_*}e^{K^{\alpha+\frac{1}{2}-\rho} \int_{T_*}^t (|u(\tau)|_{\alpha+1-\rho,\sigma}+|\bar u(\tau)|_{\alpha+1-\rho,\sigma})d\tau}.
  \eeq
\textit{Proof of this claim.}
Subtracting \eqref{u-bar-eq} from the NSE \eqref{fctnse} yields
 \beq\label{weq2}
 w'+Aw+B(w,u)+B(\bar u,w)=0 \text{ in $V'$ on $(T_2,\infty)$.}
 \eeq

Let $N\in \sigma(A)$, taking $P_N$ of \eqref{weq2} gives
  \beq\label{weq3}
 (P_Nw)'+A(P_N w)+P_N(B(w,u)+B(\bar u,w))=0 \text{ on }(T_2,\infty),
 \eeq
 in the $P_NH$-valued distribution sense.
Denote 
$$A_N=AP_N\text{ and } \tilde w_N=A^{\alpha + \frac{1}{2}-\rho} e^{\sigma A^{1/2}} P_N w=A_N^{\alpha + \frac{1}{2}-\rho} e^{\sigma A_N^{1/2}} P_N w.$$
Then it follows \eqref{weq3} that
  \beq\label{weq4}
 \tilde w_N'=-A \tilde w_N -A^{\alpha + \frac{1}{2}-\rho} e^{\sigma A^{1/2}} P_N(B(w,u)+B(\bar u,w)).
 \eeq

 In the finite dimensional space $P_NH$, we have for $t>T_*$,
 \beq
 |A\tilde w_N|\le \sqrt N |A^{1/2}\tilde w_N| \le\sqrt  N(|u(t)|_{\alpha + 1-\rho,\sigma}+|\bar u(t)|_{\alpha + 1-\rho,\sigma})
 <\sqrt N K^{-(\alpha+1/2-\rho)},
 \eeq
and, by using inequality \eqref{AalphaB},
\begin{align*}
&|A^{\alpha + \frac{1}{2}-\rho} e^{\sigma A^{1/2}} P_N(B(w,u)+B(\bar u,w))|
\le |B(w,u)|_{\alpha + \frac{1}{2}-\rho,\sigma}+|B(\bar u,w)|_{\alpha + \frac{1}{2}-\rho,\sigma}\\
&\le K^{\alpha+1/2-\rho} |w|_{\alpha + 1-\rho,\sigma}(|u|_{\alpha + 1-\rho,\sigma}+|\bar u|_{\alpha + 1-\rho,\sigma})\\
&\le |w|_{\alpha + 1-\rho,\sigma} \le |u|_{\alpha + 1-\rho,\sigma}+|\bar u|_{\alpha + 1-\rho,\sigma}\le K^{-(\alpha+1/2-\rho)}.
\end{align*}

Hence,  both $\tilde w_N$ and $\tilde w_N'$ belong to $L^\infty(T_*,\infty;P_NH)$.
 Thus, see e.g.~\cite[Ch. II, Lemma 3.2]{TemamDynBook},  equation \eqref{weq4} implies that, in the distribution sense on $(T_*,\infty)$,  
 $$ \ddt |\tilde w_N|^2=2\inprod{\tilde w_N',\tilde w_N}
 =-2\inprod{A \tilde w_N,\tilde w_N} - 2\inprod{A^{\alpha + 1/2-\rho} e^{\sigma A^{1/2}} P_N(B(w,u)+B(\bar u,w)),\tilde w_N}.$$

For $t>T_*$, applying Cauchy-Schwarz inequality and \eqref{AalphaB} yields
 \begin{align*}
  \ddt |\tilde w_N|^2 
  &\le -2|A^{1/2} \tilde w_N|^2 + 2(|B(w,u)|_{\alpha+1/2-\rho,\sigma}+|B(\bar u,w)|_{\alpha+1/2-\rho,\sigma}) |\tilde w_N|\\
  &\le -2|A^{1/2} \tilde w_N|^2 +  2K^{\alpha+1/2-\rho} |w|_{\alpha+1-\rho,\sigma}(|u|_{\alpha+1-\rho,\sigma}+|\bar u|_{\alpha+1-\rho,\sigma})|A^{1/2}\tilde w_N|.
 \end{align*}
 
In the last term,  
 $$ |w|_{\alpha+1-\rho,\sigma} \le |P_Nw|_{\alpha+1-\rho,\sigma}+|({\rm Id}-P_N)w|_{\alpha+1-\rho,\sigma} =|A^{1/2}\tilde w_N|+|({\rm Id}-P_N)w|_{\alpha+1-\rho,\sigma}.$$
 
 Thus,
  \begin{align*}
  \ddt |\tilde w_N|^2
  &\le -2\Big[1-K^{\alpha+\frac{1}{2}-\rho}(|u|_{\alpha + 1-\rho,\sigma}+|\bar u|_{\alpha + 1-\rho,\sigma})\Big]|A^{1/2}\tilde w_N|^2\\
  &\quad + 2K^{\alpha+\frac{1}{2}-\rho}(|u|_{\alpha + 1-\rho,\sigma}+|\bar u|_{\alpha + 1-\rho,\sigma})|({\rm Id}-P_N)w|_{\alpha + 1-\rho,\sigma}  |A^{1/2}\tilde w_N| .
 \end{align*}

 Then using condition \eqref{uKcond} and the relation $|A^{1/2}\tilde w_N|\ge |\tilde w_N|$ for the first term on the right-hand side, we derive for $t>T_*$,
 \beq\label{diftw}\begin{aligned}
  \ddt |\tilde w_N|^2
   &\le -2\Big[1-K^{\alpha+\frac{1}{2}-\rho}(|u|_{\alpha + 1-\rho,\sigma}+|\bar u|_{\alpha + 1-\rho,\sigma})\Big]|\tilde w_N|^2\\
     &\quad + 2|({\rm Id}-P_N)w|_{\alpha + 1-\rho,\sigma}  |A^{1/2}\tilde w_N|.
 \end{aligned}
\eeq 

By using the integrating factor, even for weak derivatives, one still obtains from \eqref{diftw} the following elementary inequality 
 \begin{align*}
    &|P_Nw(t)|_{\alpha + 1/2-\rho,\sigma}^2
  \le |P_Nw(T_*)|_{\alpha + 1/2-\rho,\sigma}^2e^{-2(t-T_*)}e^{ 2 K^{\alpha+\frac{1}{2}-\rho} \int_{T_*}^t (|u(\tau)|_{\alpha + 1-\rho,\sigma}+|\bar u(\tau)|_{\alpha + 1-\rho,\sigma})d\tau}\\
  &+2\int_{T_*}^t e^{ -2  \int_s^t \big[1-K^{\alpha+\frac{1}{2}-\rho}(|u(\tau)|_{\alpha + 1-\rho,\sigma}+|\bar u(\tau)|_{\alpha + 1-\rho,\sigma})\big]d\tau}
  |({\rm Id}-P_N)w(s)|_{\alpha + 1-\rho,\sigma}  |A^{1/2}\tilde w_N(s)|ds.
 \end{align*}

 Utilizing \eqref{uKcond} in the second summand on the right-hand side of the preceding inequality yields
 \beq\label{PNw} \begin{aligned}
    |P_Nw(t)|_{\alpha + 1/2-\rho,\sigma}^2
  &\le |P_Nw(T_*)|_{\alpha,\sigma}^2e^{-2(t-T_*)}e^{ 2 K^{\alpha+\frac{1}{2}-\rho} \int_{T_*}^t (|u(\tau)|_{\alpha + 1-\rho,\sigma}+|\bar u(\tau)|_{\alpha + 1-\rho,\sigma})d\tau}\\
  &\quad +2\int_{T_*}^t |({\rm Id}-P_N)w(s)|_{\alpha + 1-\rho,\sigma}  |A^{1/2}\tilde w_N(s)|ds.
 \end{aligned}
\eeq
 
 Observe, for all $s>T_*$,  that 
\begin{align*}
|({\rm Id}-P_N)w(s)|_{\alpha + 1-\rho,\sigma}  |A^{1/2}\tilde w_N(s)|
&\le  |w(s)|_{\alpha + 1-\rho,\sigma}^2\\
& \le (|u(s)|_{\alpha + 1-\rho,\sigma}+|\bar u(s)|_{\alpha + 1-\rho,\sigma})^2<4K^{-2(\alpha+\frac{1}{2}-\rho)}. 
\end{align*}

 We pass $N\to\infty$ in \eqref{PNw}, noticing,  by Lebesgue's  dominated convergence theorem,  that the last integral goes to zero, and obtain 
  \beqs
  |w(t)|_{\alpha+1/2-\rho,\sigma}^2
  \le |w(T_*)|_{\alpha+1/2-\rho,\sigma}^2e^{-2(t-T_*)}e^{ 2 K^{\alpha+\frac{1}{2}-\rho} \int_{T_*}^t (|u(\tau)|_{\alpha+1-\rho,\sigma}+|\bar u(\tau)|_{\alpha+1-\rho,\sigma})d\tau}.
  \eeqs
Inequality \eqref{w1} then follows.

\medskip
(c) We consider the two specified cases in \eqref{udif}.
First, we note, thanks to the expansion \eqref{uexpand} and \eqref{uu},  that
\beq\label{uubar}
|u(t)-\xi_1 t^{-1}-\xi_2 t^{-2}|_{\alpha + 1-\rho,\sigma},\ 
|\bar u(t)-\xi_1 t^{-1}-\xi_2 t^{-2}|_{\alpha + 1-\rho,\sigma} = \mathcal O(t^{-3}).
\eeq

\textit{Case $\xi_1=0$.}  Thanks to \eqref{uubar}, there are $T_*>0$ and $D_0=|\xi_2|_{\alpha + 1-\rho,\sigma}+1$ such that 
\beq\label{uu1}
|u(t)|_{\alpha + 1-\rho,\sigma},\
|\bar u(t)|_{\alpha + 1-\rho,\sigma}\le D_0 /t^2 < \frac{K^{-(\alpha+\frac{1}{2}-\rho)}}{2} \quad \forall t\ge T_*.
\eeq

Hence, condition \eqref{uKcond} is met. 
By \eqref{w1} and \eqref{uu1}, one has  for $t\ge T_*$ that
 \beqs
  |w(t)|_{\alpha + 1/2-\rho,\sigma}
\le |w(T_*)|_{\alpha + 1/2-\rho,\sigma} e^{-t+T_*}e^{K^{\alpha+\frac{1}{2}-\rho} \int_{T_*}^t 2D_0/\tau^2 d\tau}\le M_1 e^{-t},
 \eeqs
 where $M_1=|w(T_*)|_{\alpha+ 1/2-\rho,\sigma}e^{T_*+2K^{\alpha+\frac{1}{2}-\rho}D_0/T_*}$.
This proves the first relation in \eqref{udif}.

\medskip
\textit{Case $\xi_1\ne 0$.} Again, thanks to \eqref{uubar}, there are $T_*\ge 1$ and $D_0=2|\xi_1|_{\alpha + 1-\rho,\sigma}$ such that 
\beq\label{uu2}
|u(t)|_{\alpha + 1-\rho,\sigma},\
|\bar u(t)|_{\alpha + 1-\rho,\sigma}\le D_0 /t <\frac{K^{-(\alpha+\frac{1}{2}-\rho)}}{2} \quad \forall t\ge T_*.
\eeq
 Again, condition \eqref{uKcond} is satisfied, and \eqref{w1}, together with \eqref{uu2}, implies,  for $t\ge T_*$, that
 \begin{align*}
   |w(t)|_{\alpha + 1/2-\rho,\sigma}
  &\le |w(T_*)|_{\alpha + 1/2-\rho,\sigma} e^{-t+T_*}e^{K^{\alpha+\frac{1}{2}-\rho} \int_{T_*}^t 2D_0/\tau d\tau} \\
  & \le |w(T_*)|_{\alpha + 1/2-\rho,\sigma} e^{-t+T_*}e^{2K^{\alpha+\frac{1}{2}-\rho} D_0 \ln t}
  = M_2 t^\beta e^{-t},
 \end{align*}
 where $M_2=|w(T_*)|_{\alpha + 1/2-\rho,\sigma}e^{T_*}$ and $\beta=2K^{\alpha+\frac{1}{2}-\rho} D_0$.
This proves the second relation in \eqref{udif}.
\end{proof}

\begin{remark} 
(a) An equivalent condition to \eqref{xseries} is that the series $\sum_{n=1}^\infty \xi_n t^{-n}$ of expansion \eqref{uncase} converges in $G_{\alpha+1,\sigma}$ at least at one point $t=t_0\in(0,\infty)$.

(b) According to part (ii) of Theorem \ref{nextrate}, the remainder $u(t)-\bar u(t)$ cannot have any intermediate decay between the power and exponential ones. For example, it cannot be approximated by any $e^{-\mu \sqrt t}$ for $\mu\in(0,\infty)$.

\end{remark}

\bigskip
\noindent\textbf{\large Acknowledgements.} 
L.H. gratefully acknowledges the support for his research by the NSF grant DMS-1412796.

\bibliographystyle{acm}
\def\cprime{$'$}\def\cprime{$'$} \def\cprime{$'$}

\end{document}